\documentclass[12pt]{article}
\usepackage[T1]{fontenc}
\usepackage[english]{babel}
\usepackage{amsmath,amssymb,amsthm,mathrsfs}
\usepackage{amscd,bm}
\usepackage{shadethm}
\usepackage[shortlabels]{enumitem}
\usepackage{dsfont}
\definecolor{dgreen}{rgb}{0.00,0.49,0.00}
\definecolor{dblue}{rgb}{0,0.08,0.75}
\RequirePackage[colorlinks,citecolor=dgreen,urlcolor=dblue,linkcolor=dblue]{hyperref}
\definecolor{shadethmcolor}{gray}{.96}

\newcounter{ass}
\theoremstyle{plain}
\newtheorem{theorem}{Theorem}[section]
\newtheorem{lemma}[theorem]{Lemma}
\newtheorem{proposition}[theorem]{Proposition}
\newshadetheorem{sproposition}[theorem]{Proposition}
\newtheorem{corollary}[theorem]{Corollary}
\newshadetheorem{stheorem}[theorem]{Theorem}
\newshadetheorem{scorollary}[theorem]{Corollary}
\newshadetheorem{slemma}[theorem]{Lemma}
\theoremstyle{definition}
\newtheorem{definition}[theorem]{Definition}
\newtheorem{algorithm}[theorem]{Algorithm}
\newshadetheorem{sdefinition}[theorem]{Definition}
\newshadetheorem{salgorithm}[theorem]{Algorithm}
\newtheorem{remark}[theorem]{Remark}

\newtheorem{fact}[theorem]{Fact}
\newtheorem{assumption}[ass]{Assumption}

\numberwithin{equation}{section}

\newcommand{\norm}[1]{\lVert{#1}\rVert}
\newcommand{\abs}[1]{\lvert{#1}\rvert}
\newcommand{\scalarp}[1]{\langle{#1}\rangle}
\newcommand{\pair}[2]{{\langle{{#1},{#2}}\rangle}}
\newcommand{\Cll}[1]{\overline{\smash{#1}\vphantom{\phi}}}

\DeclareMathOperator*{\argmin}{argmin}

\DeclareMathOperator{\prox}{prox}

\newcommand{\HH}{\mathcal{H}}

\newcommand{\R}{\mathbb{R}}

\newcommand{\N}{\mathbb{N}}
\newcommand{\Err}{\hat{b}}

\newcommand{\EE}{\ensuremath{\mathbb E}}
\newcommand{\PP}{\ensuremath{\mathbb P}}

\begin{document}

\title{ {The iterates of FISTA converge even under\\inexact computations and stochastic gradients}}

\author{Saverio Salzo\thanks{Sapienza Universit\`a di Roma, Via Ariosto 25,  
        00185, Roma, Italy 
        ({\tt saverio.salzo@uniroma1.it})}\ \thanks{Istituto Italiano di Tecnologia, Via E.~Melen 83, 16152, Genova, Italy}}
\date{}
\maketitle

\abstract{Very recently, the papers “\emph{Point Convergence of Nesterov’s Accelerated Gradient Method: An AI-Assisted Proof}” by Jang and Ryu, and “\emph{The Iterates of Nesterov’s Accelerated Algorithm Converge in the Critical Regimes}” by Bo\c{t}, Fadili, and Nguyen  have simultaneously resolved a long-standing open problem concerning Nesterov’s accelerated gradient method. These works show that the iterates of the algorithm (known in its composite form as FISTA) indeed converge to an optimal solution. In this work, we extend these results and prove that, in infinite dimensional Hilbert spaces, the iterates of such an algorithm still converge (in the weak sense) 
even when the proximity operator and the gradient are computed inexactly, with the latter possibly stochastic. 
Our framework covers both absolute and relative
errors in proximal computations. In the stochastic setting,
it allows for uniformly bounded variance as well as 
proximal counterparts of strong growth conditions, under which
iterate convergence can be established while retaining
accelerated objective-value rates.
}

\vspace{1ex}

\allowdisplaybreaks
\section{Introduction}
In this paper, we address the composite convex optimization problem
\begin{equation}
\label{eq:mainprob}
\min_{x\in \HH} f(x)+ g(x)=:F(x),
\end{equation}
under the following assumption.
\begin{assumption}\label{ass:1}
$\HH$ is a real Hilbert space, the function $f\colon \HH \to \R$ is convex and differentiable
with Lipschitz continuous gradient, with constant $L> 0$, 
the function $g\colon \HH\to \left]-\infty,+\infty\right]$ is extended real-valued, proper, convex and lower semicontinuous, and the set of minimizers of $F:=f+g$ is nonempty.
\end{assumption}

This framework covers a wide range of problems in signal processing, machine learning, and inverse problems, and is commonly addressed using first-order proximal algorithms.

Among them, the Fast Iterative Shrinkage-Thresholding Algorithm (FISTA), introduced by Beck and Teboulle in \cite{BT} and rooted in Nesterov’s seminal acceleration technique \cite{Nesterov} and its extension to the proximal point algorithm given by G\"uler in \cite{Guler}, stands out for achieving the optimal $O(k^{-2})$ convergence rate in objective values in the convex setting. FISTA and its continuous-time analogues have been the subject of intense study in recent years (see, e.g., \cite{ApiAuDos,AttouBot,AttouCabo,AttouChba,AttouPeyp,ChamDoss,SuBoCa} and references therein), due to both their practical efficiency and by questions concerning their asymptotic behaviour.

A central question is whether the iterates themselves converge
to a minimizer. Nesterov's original analysis \cite{Nesterov}
established convergence rates for the objective values, but
did not establish convergence of the iterates.
The first substantial progress in this direction was achieved by Chambolle and Dossal \cite{ChamDoss}, who proved the convergence of the iterates 
for a modified parameter choice in the noncritical regime.
Further developments clarified the convergence properties of related inertial schemes \cite{AttouPeyp, AttouCabo,BotChen}.
However, convergence in the critical regime, which includes
the original FISTA parameter recursion, remained an open
problem for more than a decade.

Recently, two independent works, by Jang and Ryu \cite{JanRyu}, and by Bo\c{t}, Fadili, and Nguyen \cite{BotFad}, have simultaneously resolved this long-standing open problem. In \cite{JanRyu}  convergence
for the smooth finite-dimensional setting is  established, while \cite{BotFad}
 proved weak convergence in Hilbert spaces
and extended the argument to the composite setting of FISTA.\footnote{These developments followed Ryu's announcement of
convergence for the corresponding continuous-time dynamics
on X
(\url{https://x.com/ernestryu/status/1980759528984686715?s=43}).
The role of ChatGPT in the discovery of the proof is
described in \cite{JanRyu}.
}
These results settle the convergence question for exact
computations.

In practical implementations, however, gradients and proximity
operators may be evaluated only approximately, owing to
numerical errors or the use of inner iterative schemes.
The analysis of inexact accelerated proximal-gradient methods
was initiated in \cite{Villa-Salzo,Sch11}, where convergence
rates in objective values were established.
Subsequently, Aujol and Dossal \cite{AujDos} proved convergence
of the iterates under deterministic perturbations in the
noncritical regime.
Relative errors in proximal computations have also been studied,
notably in \cite{Bello,HWWM}, while broader analyses of inexact
proximal methods are developed in \cite{Barre}.
These latter contributions primarily concern objective-value
guarantees and do not establish iterate convergence 
for general convex objectives.

Stochastic gradient errors introduce a further distinction.
A uniform variance bound allows noise to persist as optimality
is approached, whereas growth conditions can require the noise
to decrease with an appropriate measure of optimality.
In the smooth convex setting, strong growth conditions allow suitable accelerated stochastic
methods to retain the $O(k^{-2})$ rate---although without any guarantees about 
the convergence of the iterates \cite{Vas19,Her25,Gup24}.
For composite problems, the nonlinear proximity operator
changes the interaction between gradient noise and the update,
while inexact proximal computations introduce additional
error terms. Extending the recent iterate-convergence results
to this setting therefore requires an analysis that accounts
for both stochastic gradient noise and proximal errors.

\subsection{Our contribution}

We prove convergence of the iterates of inexact FISTA
and of a stochastic variant in possibly infinite-dimensional
Hilbert spaces.
The analysis allows general admissible acceleration parameters,
including the critical choices.

In the deterministic setting, we allow errors in both the
gradient evaluations and the proximal computations---modeled through the inexactness criterion
of \cite{Salzo-Villa12,Sch11}.
In particular, in the critical regime, which includes the original FISTA recursion
and the choice $t_k=(k+2)/2$, 
under suitable summability conditions on the errors,
we prove weak convergence of the iterate sequence to a 
minimizer.
The main distinction from the  analysis of
Aujol--Dossal \cite{AujDos} concerns the passage from energy
estimates to convergence of the iterates.
Our proof exploits the
energy-difference argument of \cite{JanRyu}, extended here
to perturbed energies, to establish weak convergence under
more general conditions on the acceleration parameters,
including the critical regime. Coercivity of $F$ is not
required.
A detailed comparison is given in
Remark~\ref{rmk:comparisonAD}.
We also establish weak convergence under the relative
proximal-error rule considered in \cite{Bello}, subject to
the conditions stated in Remark~\ref{rmk:relproxerror}.
This extends the stability analysis to relative errors
in the critical regime.

In the stochastic setting, we consider unbiased gradient
estimates and measurable selections of inexact proximal points.
Our noise model combines a proximal counterpart of the strong
growth condition, allowing an additive noise term, with a
compatibility condition controlling the interaction between
gradient noise and proximal errors.
In the smooth case, the growth assumption reduces to the
classical strong growth condition with an additive variance
term.
Uniformly bounded variance is also covered as a special case.
Within this framework, we derive a unified objective-value
bound and prove almost-sure weak convergence of the iterates
 under appropriate conditions on the stepsizes and
proximal errors.
Moreover, when the additive term in the growth condition vanishes,
constant stepsizes and the accelerated rate $O(k^{-2})$
are possible with suitable control of the proximal errors.
For the bounded-variance setting, we obtain, up to a logarithmic factor,
the optimal rate $O(k^{-1/2})$ \cite{Lan},
with decreasing stepsizes.

The proofs combine three ingredients.
First, a recursive numerical inequality that (via a discrete version of the Bihari-La Salle Lemma)  allows to control the errors while ensuring convergence of the Lyapunov function, that need not be decreasing.
In the stochastic setting, this estimate is combined with
the Robbins--Siegmund theorem to obtain almost-sure
convergence of the energies.
Second, an averaging lemma from \cite{BotChen}, also used
in \cite{JanRyu}, transfers convergence of suitable auxiliary
scalar sequences to the iterates.
Third, a new form of Opial's lemma suitable for accelerated schemes, 
together with its stochastic extension, yields weak
convergence to a minimizer.

\subsection{Organization of the paper}

Section~\ref{subsec:notation} introduces the notation,
the inexactness criterion, the stochastic oracle assumptions,
and the two algorithms studied in the paper.
Section~\ref{sec:basic} collects the basic facts needed for the 
convergence analysis.
Section~\ref{sec:analysis} studies the admissible acceleration
parameters, establishes the preparatory lemmas, and proves
the deterministic and stochastic convergence theorems.
It also discusses relative proximal errors, the
bounded-variance specialization, and the consequences
of different parameter choices.
Appendix~\ref{sec:app} contains proofs of selected preliminary
results and of the relative-error convergence
statement in Remark~\ref{rmk:relproxerror}.

\subsection{Notation and definition of the algorithms}
\label{subsec:notation}

We denote by $\N$ the set of natural numbers, including zero, and set $\N^*=\N\setminus\{0\}$. Also, we denote by $\R_{+}=\{x\in \R\,\vert\, x\geq 0\}$
and $\R_{++}=\{x\in \R\,\vert\, x> 0\}$ the set of positive and strictly positive real numbers, respectively.
We denote by $\scalarp{\,\cdot\,,\,\cdot\,}$ and $\norm{\cdot}$ the scalar product and the norm of the Hilbert space $\HH$. We denote by $\HH\oplus \HH$ the \emph{direct sum} of $\HH$ with itself, meaning
the vector space $\HH\times \HH$ endowed with the scalar product
$((x_1,x_2), (y_1,y_2))\mapsto  \scalarp{x_1,y_1}+\scalarp{x_2, y_2}$.
Let $\varphi\colon \HH\to \left]-\infty,+\infty\right]$ be a proper convex and lower semicontinuous function. We denote by $\argmin \varphi$ the set of minimizers of $\varphi$.
Given $\varepsilon\geq0$, the $\varepsilon$-\emph{subdifferential} of $\varphi$  is the set-valued mapping
\begin{equation*}
\partial_\varepsilon \varphi\colon \HH\to 2^{\HH}\colon x\mapsto \{u\in \HH\,\vert\, (\forall\,y\in \HH)\ \varphi(y)\geq \varphi(x) + \scalarp{y-x,u}-\varepsilon\}.
\end{equation*}
When $\varepsilon=0$, we simply write $\partial \varphi$, which is called the \emph{subdifferential} of $\varphi$.
The  \emph{proximity operator} of $\varphi$ is defined as 
\begin{equation*}
\prox_{\varphi}\colon \HH\to \HH\colon x\mapsto \argmin \Big\{\varphi + \frac{1}{2}\norm{\,\cdot\, - x}^2\Big\},
\end{equation*} 
and, for every $x, z\in \HH$ we have $z = \prox_\varphi(x)\ \Leftrightarrow\ x- z \in \partial \varphi(z)$.
Moreover, the proximity operator is firmly nonexpansive, meaning that $\norm{\prox_\varphi(x)-\prox_\varphi(y)}^2\leq \scalarp{\prox_\varphi(x)-\prox_\varphi(y),x-y}$, for every $x,y\in \HH$ (see \cite{BauComb_book}).
A numerical sequence $(\alpha_k)_{k\in \N}$ is \emph{increasing} (resp.~\emph{decreasing}) if $\alpha_k \leq \alpha_{k+1}$ (resp.~$\alpha_k \geq \alpha_{k+1}$), for every $k\in \N$. 
Given two sequences of real numbers $(\alpha_k)_{k\in \N}$ and $(\beta_k)_{k\in \N}$, we write $\alpha_k \asymp \beta_k$ to denote the fact they have the same order as $k\to+\infty$, meaning that $\alpha_k = O(\beta_k)$ and $\beta_k= O(\alpha_k)$.
If $(x_k)_{k\in \N}$ is a sequence in $\HH$ and $\bar{x}\in \HH$, the \emph{strong convergence} and \emph{weak convergence} of $(x_k)_{k\in\N}$ to $\bar{x}$ is denoted by $x_k \to \bar{x}$ and $x_k \rightharpoonup \bar{x}$, respectively.
A \emph{sequential weak cluster point} of a sequence $(x_k)_{k\in\N}$ is the limit of any weakly converging subsequence. We recall that in a Hilbert space any bounded sequence always admits a weakly converging subsequence, so it has at least a sequential weak  cluster point \cite{BauComb_book}. We conclude this paragraph, by setting the notation for random variables and their (conditional) expectations. 
The underlying probability space will be denoted by $(\Omega, \mathfrak{A}, \PP)$. 
If $\HH$ is separable, $v$ is a $\HH$-valued random variable which is integrable, and $\mathfrak{F}\subset \mathfrak{A}$ is a sub-$\sigma$-algebra of $\mathfrak{A}$, the (\emph{conditional}) \emph{expectation} of $v$ (given $\mathfrak{F}$) is denoted by $\EE[v]$ (resp.~$\EE[v\,\vert\, \mathfrak{F}]$). We recall that if $w$ is another random vector which is measurable w.r.t.~$\mathfrak{F}$, then $\EE[\scalarp{v,w}\,\vert\, \mathfrak{F}] = \scalarp{\EE[v\,\vert\, \mathfrak{F}], w}$. The conditional expectation can be also defined for positive real-valued random variable, without any integrability condition.

In what follows, we present two versions of FISTA, one with deterministic errors and one with stochastic errors. Both rely on the following notion of an inexact proximity operator, introduced and analyzed in detail in \cite{Salzo-Villa12} and subsequently used, among others, in \cite{Sch11,Villa-Salzo,AujDos}.

\begin{definition}
Let $g\colon \HH\to \left]-\infty,+\infty\right]$ be a proper convex and lower semicontinuous function, $y, z\in \HH$, and $\delta\in \R_+$. Then $z \simeq_\delta \prox_{g}(y)$
means
\begin{equation*}
g(z) + \frac{1}{2}\norm{z-y}^2 \leq \min \Big\{ g + \frac 1 2 \norm{\,\cdot - y}^2\!\Big\} + \frac{\delta^2}{2}.
\end{equation*}
\end{definition}

\begin{proposition}[Lemma~2.4 \cite{Salzo-Villa12}]
\label{prop:20251109a}
Let $g\colon \HH\to \left]-\infty,+\infty\right]$ be a proper convex and lower semicontinuous function, $y, z\in \HH$, and $\delta\in \R_+$.
Let $\gamma>0$ and suppose that $z \simeq_\delta \prox_{\gamma g}(y)$. Then the following hold.
\begin{enumerate}[label={\rm (\roman*)}]
\setlength{\itemsep}{5pt} 
\smallskip
\item\label{prop:20251109a_i} $\norm{z- \prox_{\gamma g}(y)}\leq \delta$
\item\label{prop:20251109a_ii} $\exists\,\delta_1, \delta_2\geq 0,\,\exists\, e \in \HH$\ \text{s.t.}\ 
$\delta_1^2+\delta_2^2\leq \delta^2$, $\norm{e}\leq \delta_2$ and $(y + e - z)/\gamma \in \partial_{\delta_1^2/(2\gamma)} g(z)$.
\end{enumerate}
\end{proposition}

The algorithm below tolerates only deterministic errors in the computation of the gradient and of the proximity operator.

\begin{algorithm}[Inexact FISTA]
\label{fista}
Let $\gamma \in \R_{++}$ and let 
$(t_k)_{k \in \N} \in \R_+^{\N}$ be such that $t_0 = 1$,
and for every integer $k \geq 1$, $t_k \geq 1$  and $t_k^2 - t_k \leq t_{k-1}^2$. 
Let $(\delta_k)_{k \in\N}$ be a sequence in $\R_{+}$ and $(b_k)_{k\in \N}$ be a sequence in $\HH$.
Let $x_0 = y_0 \in \HH$ and define
\begin{equation}
\label{eq:fista}
\begin{array}{l}
\text{for}\;k=0,1,\ldots\\
\left\lfloor
\begin{array}{l}
x_{k+1} \simeq_{\delta_k} \prox_{\gamma g} (y_k - \gamma (\nabla f(y_k)+b_k))\\[1.5ex]
\beta_{k+1} = \dfrac{t_k - 1}{t_{k+1}}\\[1.5ex]
y_{k+1} = x_{k+1} + \beta_{k+1} (x_{k+1} - x_k).
\end{array}
\right.\\
\end{array}
\end{equation}
\end{algorithm}

Next, we present a version of FISTA that incorporates a stochastic oracle for the evaluation of the gradient of $f$, satisfying the properties stated below.
\begin{assumption}\label{ass:2}\  
\smallskip
\begin{enumerate}[label={\rm (\roman*)}]
\setlength{\itemsep}{3pt} 
\item\label{ass:2_i} $\zeta\colon \Omega\to \mathcal{Z}$ is a random variable with values in  
the measurable space $(\mathcal{Z}, \mathfrak{C})$ and defined on the
 probability space $(\Omega, \mathfrak{A}, \PP)$.
\item\label{ass:2_ii} $\HH$ is separable\footnote{The separability of $\HH$ serves for  the subsequent Lemma~\ref{StochOpial}, but also to avoid complications with the definition of measurability of vector-valued functions (see e.g.~\cite[Appendix E]{Cohn})} and $\hat{u}\colon \HH\times \mathcal{Z}\to \HH$ is a function which is measurable w.r.t.~the product $\sigma$-algebra $\mathfrak{B}(\HH)\otimes \mathfrak{C}$, where $\mathfrak{B}(\HH)$ is the Borel $\sigma$-algebra of $\HH$.
\item\label{ass:2_iii} For every $x\in \HH$, $\norm{\hat{u}(x,\zeta)}$ is square integrable and $\EE[\hat{u}(x,\zeta)] = \nabla f(x)$.
\item\!\!\!\footnote{This assumption is technical and serves to ensure measurability of the iterates whenever errors in the computation of the proximity operators occur.}\label{ass:2_iv} $\Delta\subset \R_+$, $\Gamma\subset \R_{++}$ are nonempty and, for every $(\delta, \gamma)\in \Delta\times \Gamma$, there exists a measurable map $\prox^\delta_{\gamma g}\colon \HH\to \HH$
such that $\prox^\delta_{\gamma g}(y) \simeq_\delta \prox_{\gamma g}(y), \forall\, y \in \HH$.
\item\label{ass:2_v} 
For every $\gamma\in \Gamma$,  $G_\gamma\colon \HH\to \HH$ and  $\widehat{G}_\gamma\colon \HH\times \mathcal{Z}\to \HH$ are such that
\begin{equation*}
(\forall\,y\in \HH)(\forall\,z\in \mathcal{Z})\quad
\begin{cases}
G_\gamma(y) = \dfrac{1}{\gamma}\big(y - \prox_{\gamma g}(y - \gamma \nabla f(y)) \big)\\[3ex]
\widehat{G}_\gamma(y, z) = \dfrac{1}{\gamma}\big(y - \prox_{\gamma g}(y - \gamma \hat{u}(y,z)) \big).
\end{cases}
\end{equation*}
There exists $(\rho,\sigma) \in \left[1,+\infty\right[\times \R_+$ such that, for every
$\gamma\in \Gamma$,
\begin{equation*}
(\forall\,y\in \HH)\ \ 
\EE\big[\big\langle\widehat{G}_\gamma(y,\zeta), \hat{u}(y,\zeta)- \nabla f(y)\big\rangle\big] \leq \bigg(1 - \frac{1}{\rho}\bigg) \EE\big[\Vert\widehat{G}_\gamma(y,\zeta)\rVert^2\big] + \sigma^2.
\end{equation*}
\item\label{ass:2_vi} 
There exist $\tau, \chi\in \R_+$, such that, for every $\gamma \in \Gamma$, 
\begin{equation*}
(\forall\,\delta\in \Delta)(\forall\,y\in \HH)\quad
\delta \EE[\norm{\hat{u}(y, \zeta)- \nabla f(y)}] \leq \delta \big(\tau \sqrt{\EE\big[\lVert\widehat{G}_\gamma(y, \zeta) \rVert^2\big]} +\chi\big).
\end{equation*}
\end{enumerate}
\end{assumption}

\begin{remark}\label{rmk:ass2} 
The following observations concern Assumption~\ref{ass:2}.
 \begin{enumerate}[label={\rm (\roman*)}]
 \item\label{rmk:ass2_0} In  Assumption~\ref{ass:2}\ref{ass:2_v}, $\norm{\widehat{G}(y,\zeta)}$ is square integrable. This follows from the 
 nonexpansiveness of the proximity operator and Assumption~\ref{ass:2}\ref{ass:2_iii}.
 \item\label{rmk:ass2_i} Assumption~\ref{ass:2}\ref{ass:2_v} with $\sigma=0$,
together with the unbiasedness of $\hat{u}$, implies an
interpolation property in the composite setting. 
Indeed suppose that
 $f(x) = \EE[\hat{f}(x, \zeta)]$, where $\hat{f}(\cdot, z)$ is convex and differentiable for every $z \in \mathcal{Z}$,
 and let $\hat{u}(x,z)= \nabla_1 \hat{f}(x, z)$, for every $(x,z)\in \HH\times \mathcal{Z}$.
 We will show in Lemma~\ref{lem:ass2}\ref{lem:ass2_ib} that Assumption~\ref{ass:2}\ref{ass:2_v} implies that,
for every  $x_* \in \argmin f + g$, 
\begin{equation}
\label{eq:20261018a}
\EE\big[\Vert\widehat{G}_\gamma(x_*,\zeta)\rVert^2\big] \leq \rho \sigma^2.
\end{equation}
Thus, if $\sigma=0$, 
$\widehat{G}_\gamma(x_*,\zeta)=0$ a.s., so that $x_* =\prox_{\gamma g}(x_* - \gamma \nabla_1 \hat{f}(x_*,\zeta))$ a.s.,
which in turn implies $x_*\in \argmin \hat{f}(\cdot, \zeta) + g$ a.s. 
Moreover, 
if Assumption~\ref{ass:2}\ref{ass:2_vi} also holds and $\Delta\neq \{0\}$, then \eqref{eq:20261018a}
implies
\begin{equation*}
\EE[\norm{\nabla_1 \hat{f}(x_*, \zeta)- \nabla f(x_*)}] \leq \tau \sqrt{\rho} \sigma + \chi.
\end{equation*}
Consequently, if $\sigma=\chi=0$, we obtain $\nabla_1 \hat{f}(x_*,\zeta) = \nabla f(x_*)$ a.s.
This agreement of the sample gradients at a minimizer is stronger
than composite interpolation, which only requires $-\nabla_1\hat{f}(x_*,\zeta)\in\partial g(x_*)$ a.s.
 \item Assumption~\ref{ass:2}\ref{ass:2_vi} is vacuous whenever $\Delta= \{0\}$. This corresponds to the case 
 of no errors in the computation of the proximity operators. However, when $\Delta \neq \{0\}$, the assumption is equivalent to
 \begin{equation*}
(\forall\,y\in \HH)\quad
\EE[\norm{\hat{u}(y, \zeta)- \nabla f(y)}] \leq \tau \sqrt{\EE\big[\lVert\widehat{G}_\gamma(y, \zeta) \rVert^2\big]} +\chi.
\end{equation*}
 \item Suppose that $g=0$ (smooth case). Then Assumption~\ref{ass:2}\ref{ass:2_v} is equivalent to
 \begin{equation}
 \label{eq:ass2smooth}
 (\forall\,y\in \HH)\quad
\EE[\norm{\hat{u}(y, \zeta)}^2]\leq  \rho \norm{\nabla f(y)}^2+ \rho\sigma^2.
 \end{equation}
 This follows from the fact that $G_\gamma(y)= \nabla f(y)$ and
 $\widehat{G}_\gamma(y) = \hat{u}(y, \zeta)$, and from the unbiasedness of $\hat{u}(y, \zeta)$.
Note also that, since $\EE[\norm{\hat{u}(y,\zeta) - \nabla f(y)}^2]= \EE[\norm{\hat{u}(y,\zeta)}^2]- \norm{\nabla f(y)}^2$, inequality
\eqref{eq:ass2smooth} can also be  expressed in the form
 \begin{equation}
 \label{eq:ass2smooth2}
 (\forall\,y\in \HH)\quad
\EE[\norm{\hat{u}(y, \zeta)-\nabla f(y)}^2]\leq  (\rho-1) \norm{\nabla f(y)}^2+ \rho\sigma^2.
 \end{equation}
Thus, condition \eqref{eq:ass2smooth} incorporates at the same time the \emph{strong growth condition} \cite{Vas19, Her25} (when $\rho \geq 1$ and $\sigma=0$)
and a uniform bound on the variance of the estimator $\hat{u}(y,\zeta)$ (when $\rho=1$ and $\sigma>0$).
Finally, we note that Jensen's inequality and \eqref{eq:ass2smooth2} implies that, for every $y \in \HH$,
\begin{equation*}
\EE[\norm{\hat{u}(y, \zeta)-\nabla f(y)}] \leq \sqrt{\EE[\norm{\hat{u}(y, \zeta)-\nabla f(y)}^2]} \leq \sqrt{\rho -1} \norm{\nabla f(y)} + \sqrt{\rho} \sigma.
\end{equation*}
So, Assumption~\ref{ass:2}\ref{ass:2_vi} is also satisfied with $\tau = \sqrt{\rho-1}$ and $\chi = \sqrt{\rho}\sigma$.
\item\label{rmk:ass2_v} When $g\neq 0$, another condition that more closely resembles \eqref{eq:ass2smooth2} is
\begin{equation}
\label{eq:20261014c}
 (\forall\,y\in \HH) (\forall\,\gamma\in \Gamma)\quad
\EE[\norm{\hat{u}(y, \zeta)-\nabla f(y)}^2]\leq  \beta^2 \norm{G_\gamma(y)}^2+ \kappa^2,
\end{equation}
for some constants $\beta,\kappa \in \R_+$.
In the case $\kappa =0$, with $g$ equal to the
indicator function of a nonempty closed convex set, the
corresponding condition is imposed in
\cite[Condition~3]{BKUW23}, in the context of
adaptive sampling for stochastic projected gradient methods.
We prove in Appendix~\ref{sec:app} that, if $\beta^2<1$,  condition \eqref{eq:20261014c} implies both \ref{ass:2_v} and \ref{ass:2_vi}
in Assumption~\ref{ass:2} with $\rho = (1 - \beta^2)^{-1}$, $\sigma=\kappa$, $\tau = \beta(1+\beta)$, and $\chi = \kappa (1+\beta)$.
  \end{enumerate}
\end{remark}

\begin{algorithm}[Inexact stochastic FISTA]
\label{fistastoch}
Let $(\gamma_k)_{k\in\N}$ be a sequence in $\R_{++}$ such that $\gamma_{k+1}\leq \gamma_k$, for every $k\in\N$, let  $\rho \in \left[1,+\infty\right[$ and let
$(t_k)_{k \in \N} \in \R_+^{\N}$ be such that $t_0 = 1$,
and for every integer $k \geq 1$, $t_k \geq 1$  and $t_k^2 - t_k \leq t_{k-1}^2$. 
Let $(\delta_k)_{k \in\N}$ be a sequence in $\R_{+}$.
Let Assumption~\ref{ass:2}\ref{ass:2_i}-\ref{ass:2_iv} hold and let $(\zeta_k)_{k\in \N}$ be a sequence of independent copies of $\zeta$.
Let $x_0 = y_0 \in \HH$ and define
\begin{equation}
\label{eq:stochfista}
\begin{array}{l}
\text{for}\;k=0,1,\ldots\\
\left\lfloor
\begin{array}{l}
x_{k+1} = \prox^{\delta_k}_{\gamma_k g} (y_k - \gamma_k \hat{u}(y_k, \zeta_k))\\[1.5ex]
\beta_{k+1} = \dfrac{t_k - 1}{t_{k+1}}\\[1.5ex]
y_{k+1} = x_{k+1} + \beta_{k+1} (x_{k+1} - x_k) + \dfrac{t_k}{t_{k+1}}\bigg(\dfrac{1}{\rho}-1\bigg)(x_{k+1}- y_k)
\end{array}
\right.\\
\end{array}
\end{equation}
\end{algorithm}

\section{The basic facts}
\label{sec:basic}

In this section, we present some fundamental facts upon which the main proof is built. Some of them are well-known in the related literature, while others are more recent and possibly less familiar. The first one is a special case of \cite[Lemma~2, p.~44]{Pol87}.

\begin{fact}
\label{lem:20201030b}
Let $(\alpha_k)_{k \in \N}$, $(\varepsilon_k)_{k \in \N}$ be sequences in $\R_{+}$~s.t.~$\sum_{k \in \N} \varepsilon_k <+\infty$ and
\begin{equation}
\label{eq:20201030a}
(\forall\, k \in \N)\quad  \alpha_{k+1} - \alpha_k\leq \varepsilon_k.
\end{equation}
Then $(\alpha_k)_{k \in \N}$ is convergent.
\end{fact}

\begin{remark}
In the above result, the sequence $(\alpha_k)_{k\in \N}$ can be seen as decreasing up to summable errors.
Thus, Fact~\ref{lem:20201030b} establishes that such sequences are convergent.
\end{remark}

The next fact is a discrete version of a special case of the classical Bihari--LaSalle Lemma. To the best of our knowledge it first appeared in the literature related to convex optimization in \cite{Sch11_2}, in the study of rates of convergence for the inexact FISTA---although it was not explicitly linked to the Bihari--LaSalle Lemma. In Appendix~\ref{sec:app} we provide a proof that is more direct than the one in \cite{Sch11_2} and follows the same ideas of the continuous version.

\begin{fact}[Bihari--LaSalle Lemma, discrete version]
\label{lem:Bih-LaS}
Let $(\mu_k)_{k \in \N}$ and $(\lambda_k)_{k \in \N}$ be two sequences in $\R_+$ and $(\sigma_k)_{k \in \N}$
be an increasing sequence in $\R_+$ such that
\begin{equation*}
(\forall\,k\in \N)\ \mu_k \leq \sigma_k + \sum_{i=0}^k \lambda_i \sqrt{\mu_i}.
\end{equation*}
Then, for all $k\in \N$, $\displaystyle \max_{0\leq i\leq k} \sqrt{\mu_{i}}\leq \frac 1 2 \sum_{i=0}^k \lambda_i+ \bigg( \Big(\frac1 2 \sum_{i=0}^k \lambda_i\Big)^2 + \sigma_k\bigg)^{\!1/2}$.
\end{fact}

The following result is a standard formulation of Opial’s Lemma \cite{Opial1967,AraBau}. We include its proof here (which is reported from \cite[Lemma~2.47]{BauComb_book}), since in the subsequent remark we slightly weaken the assumptions to make it more useful for our purposes. 
Later, in Lemma~\ref{StochOpial} we will give also a stochastic version. This is one of the points in which the analysis of convergence of the iterates of Nesterov's accelerated method diverges from the classical analysis of fixed-point iterations of a nonexpansive operator---which includes that of the (nonaccelerated) proximal gradient algorithm. 

\begin{fact}[Opial's Lemma]
\label{lem:20160601}
Let $S \subset \HH$ be nonempty.
Let $(x_k)_{k \in \N}$ be a sequence in $\HH$ and suppose that
the sequential weak cluster points of $(x_k)_{k \in \N}$ belong to $S$ and that
 for any $y \in S$, $(\norm{x_k - y})_{k \in \N}$
is convergent. Then $(x_k)_{k \in \N}$
weakly converges to a point in $S$.
\end{fact}
\begin{proof}
The assumptions ensure that $(x_k)_{k \in \N}$ is bounded. 
Therefore, the set of sequential weak cluster points of
 $(x_k)_{k \in \N}$ is nonempty. Let $y_1, y_2 \in \HH$ 
 be two sequential weak cluster points of $(x_k)_{k \in \N}$, so that there exist
 two subsequences $(x_k^{1})_{k \in \N}$ and $(x_k^{2})_{k \in \N}$, of $(x_k)_{k \in \N}$, such that $x_k^1 \rightharpoonup y_1$
and $x_k^2 \rightharpoonup  y_2$. Then, for every $k \in \N$,
\begin{equation}
\label{e:20150505e}
\norm{x_k - y_1}^2- \norm{x_k - y_2}^2 = 2 \scalarp{x_k, y_2 - y_1} + \norm{y_1}^2 - \norm{y_2}^2.
\end{equation}
Since $y_1$ and $y_2$ are sequential weak cluster points of $(x_k)_{k \in \N}$, by assumptions, 
 $y_1, y_2 \in S$ and $(\norm{x_k - y_1})_{k \in \N}$ and $(\norm{x_k - y_2})_{k \in \N}$ 
are convergent. 
Therefore by \eqref{e:20150505e}, we obtain
that there exists $\beta \in \R$ such that $\scalarp{x_k,y_2 - y_1} \to \beta$.
Now, since $x_k^i \rightharpoonup y_i$, $i=1,2$, we have
$\scalarp{x_k^i, y_2 - y_1} \to \scalarp{y_i, y_2 - y_1}$, which implies
\begin{equation*}
\scalarp{y_1, y_2 - y_1} = \beta = \scalarp{y_2, y_2 - y_1}
\end{equation*}
and hence $\norm{y_2 - y_1}^2 = 0$. This proves that the set of sequential weak cluster points of the sequence $(x_k)_{k \in \N}$
is a singleton contained in $S$. So, the sequence $(x_k)_{k \in \N}$ is weakly convergent to a point in $S$. 
\end{proof}

\begin{remark}
\label{rmk:opial}
It is clear from the proof above that
Opial's Lemma remains true if the condition that $(\norm{x_k - y}^2)_{k \in \N}$
is convergent for any $y\in S$, is replaced by the following two:
\smallskip
\begin{itemize}
\setlength{\itemsep}{3pt} 
\item $(x_k)_{k \in \N}$ is bounded
\item $\forall\,y_1, y_2\in S\colon (\scalarp{x_k, y_2- y_1})_{k \in \N}$ is convergent.
\end{itemize}
\end{remark}

The next result is the second pivotal ingredient to handle the convergence of the iterates of Nesterov's accelerated method.

\begin{fact}[\cite{BotChen}, Lemma~A.2]
\label{lem:radu_enis}
Let $(a_k)_{k\in \N}$ be sequence of real numbers and let $(\lambda_k)_{k\in \N}$
be a sequence in $\R_{++}$. Let $\ell\in \R$ and suppose that $\sum_{k=0}^{+\infty} 1/\lambda_k = +\infty$. Then
\begin{equation*}
a_{k+1}+ \lambda_k(a_{k+1}-a_k)\to \ell\ \Rightarrow\ a_k\to\ell.
\end{equation*}
\end{fact}

We conclude by presenting the classical Robbins--Siegmund theorem on almost supermartingales, that will be at the basis of the analysis of the stochastic version of FISTA.

\begin{fact}[Robbins--Siegmund Theorem \cite{RobSig1971}]
\label{thm:RiSi71}
Let $(\mathfrak{F}_k)_{k \in \N}$ be a filtration on the probability space $(\Omega, \mathfrak{A}, \PP)$ and let
 $(\xi_k)_{k \in \N}$, $(\chi_k)_{k \in \N}$, $(\varepsilon_k)_{k \in \N}$, and $(\zeta_k)_{k \in \N}$
 be sequences of positive real-valued random variables which are adapted to $(\mathfrak{F}_k)_{k \in \N}$. Suppose that 
 \begin{enumerate}[label={\rm (\roman*)}, leftmargin=0.65cm]
\smallskip
 \setlength{\itemsep}{3pt} 
 \item\label{thm:RiSi71_i} $\sum_{k=0}^{+\infty} \chi_k<+\infty$ and $\sum_{k=0}^{+\infty}\varepsilon_k<+\infty$ a.s.
 \item\label{thm:RiSi71_ii} for every $k \in \N$, $\EE[\xi_{k+1}\,\vert\, \mathfrak{F}_k] \leq (1+\chi_k)\xi_k + \varepsilon_k - \zeta_k$.
 \end{enumerate}
 \smallskip
 Then there exists a positive real-valued random variable $\xi$ such that $\xi_k \to \xi$ a.s.~and $\sum_{k=0}^{+\infty} \zeta_k<+\infty$ a.s.
\end{fact}

\section{Analysis of the convergence}
\label{sec:analysis}

Throughout the rest of the paper we make Assumption~\ref{ass:1} and we set
\begin{equation*}
F_* = \min_{x\in \HH} F(x)\quad\text{and}\quad S_*= \argmin_{x\in \HH} F(x).
\end{equation*}

\subsection{The acceleration parameters}

We start by clarifying the condition 
\begin{equation}
\label{eq:20251107a}
t_0=1\quad\text{and}\quad\forall\,k\in \N,\,k\geq 1\colon\ \  
1\leq t_k\ \ \text{and}\ \  t_k^2 - t_k - t_{k-1}^2 \leq 0,
\end{equation}
on the parameters $t_k$'s in Algorithms~\ref{fista} and \ref{fistastoch}. A numerical sequence $(t_k)_{k\in \N}$ satisfying \eqref{eq:20251107a} will be called a \emph{sequence of admissible parameters for} FISTA. 

We note that, by solving the above quadratic inequality in $t_k$, condition \eqref{eq:20251107a} can be equivalently written as
\begin{equation}
\label{eq:20251113a}
t_0=1\quad\text{and}\quad\forall\,k\in \N^*\colon\ \ 1\leq t_k\leq  \frac{1 + \sqrt{1 + 4 t_{k-1}^2}}{2}=:\varphi(t_{k-1}),
\end{equation}
where $\varphi\colon \left[1,+\infty\right[ \to \left[1,+\infty\right[\colon t\mapsto \big(1 + \sqrt{1 + 4 t^2}\big)/2$. It is easy to see that the function $\varphi$ satisfies the following properties
\begin{enumerate} [label={\rm (\roman*)}]
\smallskip
\setlength{\itemsep}{3pt} 
\item $\forall\,t \geq 1$, $\varphi(t)-t > 1/2$
\item $\varphi(t)-t \downarrow 1/2$ as $t\uparrow +\infty$.
\end{enumerate}

\vspace{1ex}
Below we collect the properties of sequences of admissible parameters for FISTA.
\begin{proposition}
\label{prop:20251112a}
Let $(t_k)_{k\in \N}$ be a sequence in  $\R_+$. The following hold.
\begin{enumerate}[label={\rm(\roman*)}]
\smallskip
\setlength{\itemsep}{3pt} 
\item\label{prop:20251112a_0} If $(t_k)_{k\in \N}$ is a sequence of admissible parameters for {\rm FISTA},
then $t_k \leq t_{k-1} + 1$, and hence $t_k\leq 1+k$, for every integer $k\geq 1$.
\item\label{prop:20251112a_i} $(t_k)_{k\in \N}$ is an increasing sequence of admissible parameters for {\rm FISTA} iff
\begin{equation}
\label{eq:20251112a1}
t_0=1\quad\text{and}\quad\forall\,k\in\N^*\colon t_k \in [t_{k-1}, \varphi(t_{k-1})].
\end{equation}
Also, if \eqref{eq:20251112a1} holds and $t_1>1$, then
 $0<\beta_k<1$, for all integers $k\geq 2$, and $1- \beta_k \asymp 1/t_k$.
\item\label{prop:20251112a_ii} Given $a \in \left[2,+\infty\right[$, suppose that $(t_k)_{k \in \N}$ satisfies the stronger condition\footnote{Recall that the length of the interval $[t_{k-1}, \varphi(t_{k-1})]$ is strictly greater than $1/2$.} 
\begin{equation}
\label{eq:20251112b}
t_0=1\quad\text{and}\quad\forall\,k\in \N^*\colon t_k \in [t_{k-1}+1/a, \varphi(t_{k-1})].
\end{equation}
Then, $\forall\,k\in \N^*$, $t_k - t_{k-1}\geq 1/a$ and hence $\frac{k+a}{a}\leq t_k \leq k+1$,
so that $t_k \asymp k$.
\item\label{prop:20251112a_iii} Given $\alpha \in \left]0,1\right[$, if $(s_k)_{k \in \N}$ is a sequence of admissible parameters for {\rm FISTA}, then $(t_k)_{k\in \N} \equiv (s_k^\alpha)_{k\in \N}$ 
is a sequence of admissible parameters for {\rm FISTA} too.
\end{enumerate}
\end{proposition}
\begin{proof}
\ref{prop:20251112a_0}: If $(t_k)_{k\in\N}$ is a sequence of admissible parameters for FISTA, it follows from \eqref{eq:20251113a} and the fact that $\sqrt{a+b}\leq \sqrt{a}+\sqrt{b}$, for any $a,b\in \R_+$, that
\begin{equation*}
\forall\,k\in \N^*\colon t_k \leq \frac{1 + 1 + 2 t_{k-1}}{2} = 1 + t_{k-1}.
\end{equation*}
Hence, $t_k - 1 = \sum_{i=1}^k (t_i - t_{i-1}) \leq k $, which yields that
$t_k \leq 1+k$, for every $k \in \N$.

\ref{prop:20251112a_i}: The first part is immediate since in \eqref{eq:20251112a1}, the condition $t_k\geq t_{k-1}$
says that $(t_k)_{k\in\N}$ is increasing while we already noted that the condition $t_k \leq \varphi(t_{k-1})$ is equivalent  to the quadratic inequality in \eqref{eq:20251107a}. The second part of the statement follows from the following inequality: for every integer $k\geq 1$
\begin{equation*}
\frac{1}{t_{k}}\leq \underbrace{\frac{t_{k}- t_{k-1}+1}{t_{k}}}_{1-\beta_k} \leq \frac{2}{t_{k}}.
\end{equation*}

\ref{prop:20251112a_ii}: By the second condition in \eqref{eq:20251112b}, it is clear that  $t_k - t_{k-1}\geq 1/a$ and hence that $t_k - 1 = \sum_{i=1}^k (t_{i}- t_{i-1}) \geq k /a$.

\ref{prop:20251112a_iii}: Let $\alpha \in \left]0,1\right[$ and suppose that, $\forall\,k\in\N^*$, $s_k^2 - s_k - s_{k-1}^2 \leq 0$ and $s_0=1$. Then, for every integer $k\geq 1$
\begin{equation*}
s_k(s_k- 1)\leq s_{k-1}^2\ \Rightarrow\ s_k^\alpha(s_k- 1)^\alpha\leq s_{k-1}^{2\alpha}\ \Rightarrow\ s_k^\alpha(s_k^\alpha-1)\leq (s_{\textcolor{red}{k-1}}^{\alpha})^2,
\end{equation*}
where we used that, for every $u>1$, $(u-1)^\alpha > u^\alpha - 1$ (which in turn follows from the fact that $w\mapsto w^\alpha$ is concave on $\left[0,+\infty\right[$ and vanishes at zero, hence subadditive). 
\end{proof}

\begin{remark}\ 
\begin{enumerate}[label={\rm(\roman*)}]
\item\label{rmk:FISTAparam_i} Two special increasing sequences of admissible parameters for FISTA are
\begin{equation}
\label{eq:20251107d}
t_k = \frac{k+a}{a}\quad(a\geq 2)\quad \text{and}\quad t_k = \frac{1 + \sqrt{1 + 4 t_{k-1}^2}}{2}\ (\text{with}\ t_0=1),
\end{equation}
and, in the second case, $(k+2)/2\leq t_k \leq k+1$, for every $k\in \N$.
\item Proposition~\ref{prop:20251112a}\ref{prop:20251112a_ii}-\ref{prop:20251112a_iii} shows that, for any $\alpha \in [0,1]$, it is always possible to define a sequence $(t_k)_{k\in \N}$ which is admissible for FISTA and satisfies
 $t_k\asymp k^\alpha$, so that $1-\beta_k \asymp 1/k^\alpha$. In particular, in view of \ref{rmk:FISTAparam_i}, we can choose, for every $k\in\N$,
\begin{equation}
\label{eq:20251112d}
t_k = \bigg(\frac{k+2}{2}\bigg)^{\!\alpha}\ \text{or}\ s_k = \frac{1+ \sqrt{1 + 4 s^2_{k-1}}}{2}\ (\text{with}\ s_0=1)\ \text{and}\ t_k = s_k^\alpha.
\end{equation}
\item If $(t_k)_{k\in \N}$ is a sequence of admissible parameters for FISTA which is increasing, then we have that $t_k \to \bar{t}\in [1,+\infty]$ and three possibilities may occur: (1) $\bar{t}=1$, which means that $t_k = t_0=1$, for every $k\in \N$, so that $\beta_{k+1}=0$, turning Algorithm~\ref{fista} into the  classical (nonaccelerated) proximal gradient algorithm; (2) $1<\bar{t}<+\infty$,
and in such case $\beta_{k} \to (\bar{t}-1)/\bar{t}<1$; (3) $\bar{t}= +\infty$, which yields, to  Proposition~\ref{prop:20251112a}\ref{prop:20251112a_i}, $\beta_k\to 1$.
\end{enumerate}
\end{remark}

\subsection{Three preparatory lemmas}

We now present a perturbed version of the basic inequality of the proximal gradient descent algorithm. 
This result was implicitly given in \cite{Sch11_2} and more explicitly in \cite{AujDos} (although with $\gamma\leq 1/L$).

\begin{lemma}
\label{lem:20190313d}
Let $f\colon \HH\to \R$ be convex and Lipschitz smooth with constant $L>0$ and let $g\colon \HH\to \left]-\infty, +\infty\right]$ be proper convex and lower semicontinuous. Let $x, x_+, b\in \HH$, $\gamma\in \R_{++}$, $\delta \in \R_+$, and suppose that
\begin{equation*}
x_+ \simeq_\delta \prox_{\gamma g} (x - \gamma (\nabla f(x)+b)).
\end{equation*}
Then, there exists $\delta_1,\delta_2 \geq 0$ and $e\in \HH$ such that $\norm{e}\leq \delta_2$,
$\delta_1^2 + \delta_2^2 \leq \delta^2$
and
\begin{align*}
(\forall\,z\in \HH)\quad 2\gamma (F(x_+)- F(z)) &\leq \norm{x-z}^2 - \norm{x_+-z}^2 - (1-\gamma L)\norm{x-x_+}^2\\[1ex]
&\hspace{20ex} + 2\pair{x_+-z}{e - \gamma b}+\delta_1^2.
\end{align*}
\end{lemma}
\begin{proof}
By Proposition~\ref{prop:20251109a}\ref{prop:20251109a_ii} we derive that $\exists\,\delta_1, \delta_2\geq 0$ and $\exists\,e \in \HH$ such that 
\begin{equation*}
\norm{e}\leq \delta_2\quad\text{and}\quad \frac{x - x_+}{\gamma} - \nabla f(x)+ \frac{e}{\gamma} -b\in \partial_{\delta_1^2/(2\gamma)}g(x_+).
\end{equation*}
Hence
\begin{equation}
\label{eq:20261014a}
g(x_+)- g(z) \leq \Big\langle x_+-z, \frac{x - x_+}{\gamma} - \nabla f(x)+ \frac{e}{\gamma} -b\Big\rangle + \frac{\delta_1^2}{2 \gamma}.
\end{equation}
On the other hand, convexity and Lipschitz smoothness of $f$ yields
\begin{align}
\label{eq:20261014b}
f(x_+) - f(z) & = f(x_+) - f(x) + f(x) - f(z)\notag\\
& \leq \scalarp{x_+ - x, \nabla f(x)} + \frac{L}{2}\norm{x_+ - x}^2 + \scalarp{x - z, \nabla f(x)}\notag\\
& = \scalarp{x_+ - z, \nabla f(x)} +\frac{L}{2}\norm{x_+ - x}^2.
\end{align}
Summing \eqref{eq:20261014a} and \eqref{eq:20261014b} we obtain
\begin{equation*}
F(x_+) - F(z) \leq \frac{1}{\gamma}\scalarp{x_+ - z, x - x_+} +\frac{L}{2}\norm{x_+ - x}^2 + \Big\langle x_+ - z, \frac{e}{\gamma}- b\Big\rangle 
+ \frac{\delta_1^2}{2 \gamma}.
\end{equation*}
Finally, multiplying both sides by $2\gamma$ and using the identity
$\norm{x- z}^2 = \norm{x - x_+}^2 + \norm{x_+-z}^2 +2\pair{x-x_+}{x_+-z}$, the statement follows.
\end{proof}

Next, we prove a general bound concerning a recursive inequality that arises
in the convergence analysis of Algorithm~\ref{fista} and Algorithm~\ref{fistastoch}. This result will allow to control the  errors along the iterations while keeping the underlying Lyapunov function converging, although it is no longer decreasing.
\begin{lemma}
\label{lm:recurrences}
Let $(\alpha_k)_{k \in \N}, (\lambda_k)_{k \in \N}$ and $(\xi_k)_{k \in \N}$ be sequences in $\R_+$ such that
\begin{align}
\label{eq:recurrences}
\forall k \in \N\colon  \alpha_{k+1} \leq \alpha_k + \lambda_k \sqrt{\alpha_{k+1}} + \xi_k.
\end{align}
Then the following hold for every $k \in \N$
\begin{equation}
\label{eq:20261012a}
\max_{0\leq i\leq k+1}\alpha_{i} 
\leq \left(\sqrt{\alpha_0} +\sum_{i=0}^k \lambda_i + \sqrt{\sum_{i=0}^k \xi_i}\,\right)^{\!\!2}.
\end{equation}
Moreover, if $(\lambda_k)_{k\in\N}$ and $(\xi_k)_{k\in\N}$ are summable,
then $(\alpha_k)_{k \in \N}$ is convergent.
\end{lemma}
\begin{proof}
By summing the inequality $\alpha_{i+1}-\alpha_i \leq \lambda_i \sqrt{\alpha_{i+1}}+\xi_i$ from $i=0$ to $i=k$, 
and then setting $\sigma_k = \alpha_0 + \sum_{i=0}^k \xi_i$, we obtain
\begin{equation*}
\forall k\in \N\colon\ \alpha_{k+1} \leq \sigma_k + \sum_{i=0}^k \lambda_i \sqrt{\alpha_{i+1}}.
\end{equation*}
Hence, by the Bihari--LaSalle Lemma (Fact~\ref{lem:Bih-LaS}) applied to the sequence $\mu_k = \alpha_{k+1}$, 
\begin{equation*}
\max_{0\leq i\leq k} \sqrt{\alpha_{i+1}}\leq \frac 1 2 \sum_{i=0}^k \lambda_i + \sqrt{\bigg(\frac1 2 \sum_{i=0}^k \lambda_i\bigg)^{2}+ \alpha_0 + \sum_{i=0}^k \xi_i}.
\end{equation*}
Then, \eqref{eq:20261012a} follows by applying the inequality $\sqrt{a+b+c}\leq \sqrt{a}+\sqrt{b}+\sqrt{c}$,
for every $a,b,c \in \R_+$, and then squaring both sides.
Concerning the final part of the statement we note that, thanks to inequality \eqref{eq:20261012a},
if $(\lambda_k)_{k \in \N}$ and $(\xi_k)_{k \in \N}$ are summable, then
$(\sqrt{\alpha_{k+1}})_{k \in \N}$ is bounded, say by $M\geq 0$ and hence the recursive inequality
\eqref{eq:recurrences} yields $\alpha_{k+1} \leq \alpha_k + (\lambda_k M + \xi_k)$ and the statement follows from 
Fact~\ref{lem:20201030b}.
\end{proof}

We conclude the section by giving a stochastic version of the Opial's Lemma 
that will be critical for the analysis of Algorithm~\ref{fistastoch}.
This result is a generalization of \cite[Proposition~2.3]{ComPes}. See also \cite{Rav24} for similar ideas.

\begin{lemma}[Stochastic Opial's Lemma]\label{StochOpial}
Let $\HH$ be a separable real Hilbert space, let $S\subset \HH$ be nonempty and let $(x_k)_{k\in\mathbb{N}}$ 
be a sequence of  random vectors on $\left(\Omega, \mathfrak{A}, \PP\right)$ with values in $\HH$.
Consider the following conditions.
\begin{enumerate}[label={\rm (\alph*)}]
\setlength{\itemsep}{3pt} 
\smallskip
\item\label{StochOpial_i} $(x_k)_{k\in\mathbb{N}}$ is bounded a.s.
\item\label{StochOpial_ii} for every $y_1, y_2\in S$, there exists $\Omega_{y_1, y_2}\in \mathfrak{A}$, with $\PP(\Omega_{y_1, y_2})=1$,
such that, for every $\omega\in\Omega_{y_1, y_2}$, the sequence $(\scalarp{x_k(\omega), y_1- y_2})_{k\in \N}$ converges.
\item\label{StochOpial_iii} the set of sequential weak cluster points of $(x_k)_{k\in \N}$ is contained in $S$ almost surely.
\end{enumerate}
\smallskip
Suppose that  \ref{StochOpial_i} and \ref{StochOpial_ii} hold.
Then, there exists $\tilde{\Omega}\in \mathfrak{A}$, with $\PP(\tilde{\Omega})=1$, such that, for every $\omega\in \tilde{\Omega}$, $(x_k(\omega))_{k\in \N}$ is bounded and
for every $y_1, y_2\in S$, 
$(\scalarp{x_k(\omega), y_1- y_2})_{k\in \N}$
is convergent. Moreover, if in addition condition \ref{StochOpial_iii} holds,
then there exists an $S$-valued random vector $\bar{x}$ such that $x_k\rightharpoonup \bar{x}$ a.s.
\end{lemma}
\begin{proof}
We prove the first part of the statement. Suppose that \ref{StochOpial_i} and \ref{StochOpial_ii} hold.
Since $S$ is a subset of a separable metric space, it is separable.
Let $D\subset S$ be countable and dense in $S$ and let $\tilde{\Omega}_1=\bigcap_{z_1,z_2\in D} \Omega_{z_1,z_2}$. 
Then $\tilde{\Omega}_1\in \mathfrak{A}$ and $\PP(\tilde{\Omega}_1)=1$, and, for every $\omega\in\tilde{\Omega}_1$ and for every $z_1,z_2\in D$, 
\begin{equation*}
(\scalarp{x_k(\omega), z_1 - z_2})_{k\in \N}\ \text{is convergent.}
\end{equation*}
Let $\tilde{\Omega}_2\subset \Omega$ be such that $\PP(\tilde{\Omega}_2)=1$ and 
for every $\omega\in \tilde{\Omega}_2$, $(x_k(\omega))_{k\in \N}$ is bounded.
Fix $\omega\in\tilde{\Omega}:= \tilde{\Omega}_1\cap \tilde{\Omega}_2$ and let, for every $k \in \N$, $f_k\colon \HH\oplus\HH\to \R\colon (y_1, y_2) \mapsto \scalarp{x_k(\omega), y_1 - y_2}$. Clearly $f_k$ is linear and continuous and, since $(x_k(\omega))_{k\in \N}$ is bounded, we have
\begin{equation}
\label{eq:20251125a}
(\forall\,k\in \N)(\forall\, (y_1,y_2)\in \HH\oplus\HH)\ \ 
\abs{f_k(y_1,y_2)} \leq M\norm{(y_1, y_2)},
\end{equation}
for some $M>0$.
Thus, $(f_k(y_1,y_2))_{k\in \N}$ is bounded, for every $(y_1,y_2)\in \HH\oplus\HH$, and hence the functions $\liminf_k f_k$ and $\limsup_k f_k$ are real valued.
We show that the set
\begin{equation*}
C= \big\{(y_1,y_2) \in \HH\oplus\HH \,\big\vert\, \liminf_k f_k(y_1,y_2) = \limsup_k f_k(y_1,y_2) \big\}
\end{equation*}
is closed. 
In order to prove such statement, we observe that, by \eqref{eq:20251125a} we have,
for every $(y_1, y_2), (z_1, z_2)\in \HH\oplus \HH$,
\begin{equation}
\label{eq:20250812a}
\abs{f_k(y_1,y_2)-f_k(z_1,z_2)} \leq M\norm{(y_1,y_2)- (z_1, z_2)}.
\end{equation}
Thus, it follows that, for every $(y_1, y_2), (z_1, z_2)\in \HH\oplus \HH$,
\begin{equation*}
f_k(z_1,z_2)- M\norm{(y_1,y_2)- (z_1, z_2)} \leq f_k(y_1,y_2) \leq f_k(z_1,z_2) + M\norm{(y_1,y_2)- (z_1, z_2)}
\end{equation*}
and hence, taking $\liminf$,
\begin{align*}
\liminf_k f_k(z_1,z_2) &- M\norm{(y_1,y_2)- (z_1, z_2)}\\ 
&\leq \liminf_k f_k(y_1,y_2) 
\leq \liminf_k f_k(z_1,z_2) + M\norm{(y_1,y_2)- (z_1, z_2)}.
\end{align*}
Therefore, we obtain that, for every $(y_1, y_2), (z_1, z_2)\in \HH\oplus \HH$, 
\begin{equation*}
\big\lvert\liminf_k f_k(y_1,y_2)- \liminf_k f_k(z_1,z_2)\big\rvert\leq M\norm{(z_1,z_2)-(y_1,y_2)},
\end{equation*}
meaning that $\liminf_k f_k$ is Lipschitz continuous and hence continuous. Similarly $\limsup_k f_k$,
is Lipschitz continuous 
and the closedness of $C$ follows. 
Now, since $C$ is closed, we clearly have $D\times D\subset C\ \Rightarrow\ S\times S \subset \Cll{D}\times \Cll{D}\subset C$ and the first claim in the statement follows.
Concerning the second part of the statement,
if we additionally assume that \ref{StochOpial_iii} holds, then there exists 
$\tilde{\Omega}_3\in \mathfrak{A}$, with $\PP(\tilde{\Omega}_3)=1$ such that,
for every $\omega \in \tilde{\Omega}_3$, the set of sequential weak cluster points of $(x_k(\omega))_{k\in \N}$ is contained in $S$. Thus, if we take $\omega \in \tilde{\Omega} \cap \tilde{\Omega}_3$, we have that the sequence $(x_k(\omega))_{k\in \N}$
satisfies all the assumptions required by the deterministic Opial's Lemma (Fact~\ref{lem:20160601}) in the form mentioned in Remark~\ref{rmk:opial}. Therefore, $(x_k(\omega))_{k\in \N}$ weakly converges to a point in $S$ and the statement follows.\footnote{Fix $x_*\in S$. The map $\bar{x}\colon \Omega\to \HH$, defined by $x_k(\omega)\rightharpoonup \bar{x}(\omega)$, if $\omega \in \tilde{\Omega} \cap \tilde{\Omega}_3$, and $x_*\in S$ otherwise is weakly measurable, and hence, being $\HH$ separable, measurable (see e.g.~\cite[Theorem E.9]{Cohn}).}
\end{proof}

\subsection{The two main theorems}

Finally, we are ready to provide the first main result of the paper.

\begin{theorem}
\label{mainthm}
Under Assumption~\ref{ass:1},
let $(x_k)_{k \in \N}$ be defined as in Algorithm~\ref{fista} with $\gamma \leq 1/L$.
 Then, for every $x_*\in S_*$ and every integer $k\geq 1$
 \begin{equation*}
 F(x_k) - F_*  \leq  \frac{1}{2\gamma t^2_{k-1}} 
 \left[ \norm{x_0- x_*} + 
 2 \sum_{i=0}^{k-1} t_i \big(\delta_{2,i} + \gamma\norm{b_i}\big)
 + \sqrt{\sum_{i=0}^{k-1} t_i^2 \delta_{1,i}^2}\,\right]^2,
 \end{equation*}
where $(\delta_{1,k})_{k\in \N}$ and $(\delta_{2,k})_{k\in \N}$ are two sequences in $\R_+$ such that,
 $\delta_{1,k}^2 + \delta_{2,k}^2 \leq \delta_{k}^2$, for every $k\in \N$, defined according to Proposition~\ref{prop:20251109a}.
 Moreover, suppose that
$t_k\to +\infty$, and that $(t_k \delta_k)_{k\in\N}$ and $(t_k \norm{b_k})_{k\in \N}$ are summable.
 Then  $F(x_k)- F_* = O(1/t^2_{k-1})$ and
 $\exists\,x_*\in S_*$
such that $x_k\rightharpoonup x_*$ and $y_k\rightharpoonup x_*$.
\end{theorem}

\begin{proof}
It follows from the definition of $y_{k+1}$ in Algorithm~\ref{fista} that,
for every $k \in \N$,
\begin{equation*}
y_{k+1} = \bigg(1 - \frac{1}{t_{k+1}} \bigg) x_{k+1} + \frac{1}{t_{k+1}}
\big(\underbrace{x_k + t_k(x_{k+1} - x_k)}_{=:v_{k+1}}\big).
\end{equation*}
Therefore, for every $k \in \N$,
\begin{equation}
\label{eq:20190313b}
y_k = \bigg(1 - \frac{1}{t_k} \bigg) x_k + \frac{1}{t_k} v_k \qquad(v_0:= y_0).
\end{equation}
Let $k\in \N$. The definition of $v_{k+1}$ yields that
$v_{k+1} - x_k = t_k(x_{k+1} - x_k)$ and hence
\begin{equation}
\label{eq:20251104b}
x_{k+1} = x_k + \frac{1}{t_k}(v_{k+1} - x_k) = \bigg(1 - \frac{1}{t_k} \bigg) x_k 
+ \frac{1}{t_k} v_{k+1}.
\end{equation}
Applying Lemma~\ref{lem:20190313d}, with $x = y_k$ (noting that $\gamma L\leq 1$), we obtain
\begin{multline}
\label{eq:20190313e}
(\forall\, z \in \HH)\quad 
2 \gamma (F(x_{k+1})- F(z))  \\[0.5ex]
\leq  \norm{y_k-z}^2 -  \norm{x_{k+1} - z}^2
+ 2\langle x_{k+1}-z, e_k -  \gamma b_k\rangle+\delta_{1,k}^2,
\end{multline}
where $\delta_{1,k}, \delta_{2,k}, \delta_{k} \geq 0$ and $e_k\in \HH$ are such that
$\delta_{1,k}^2+ \delta_{2,k}^2 \leq \delta_{k}^2$ and $\norm{e_k}\leq \delta_{2,k}$.
Now, let $x_* \in S_*$ and set
\begin{equation}
\label{eq:20261009e}
z = \bigg(1 - \frac{1}{t_k} \bigg) x_k 
+ \frac{1}{t_k} x_*.
\end{equation}
Then, we derive from \eqref{eq:20190313b} and \eqref{eq:20251104b} that
\begin{equation}
\label{eq:20261014a2}
y_k - z = \frac{1}{t_k}(v_k - x_*)\quad\text{and}\quad x_{k+1} - z = \frac{1}{t_{k}} (v_{k+1} - x_*).
\end{equation}
Next, set $r_k = F(x_k)- F_*$. Then, by the convexity of $F$ and \eqref{eq:20261009e} we have
\begin{equation}
\label{eq:20261014b2}
F(z) - F_* \leq \bigg(1 - \frac{1}{t_k}\bigg)r_k.
\end{equation}
Therefore, using \eqref{eq:20261014a2} and \eqref{eq:20261014b2}  into \eqref{eq:20190313e}, recalling that $\norm{e_k}\leq \delta_{2,k}$, and rearranging the terms, we obtain
\begin{align*}
2 \gamma r_{k+1} &+ \frac{\norm{v_{k+1} - x_*}^2}{t_k^2} \\
&\leq \bigg(1 - \frac{1}{t_k} \bigg) 2 \gamma r_k + \frac{\norm{v_k - x_*}^2}{t_k^2}
+ \frac{ 2 (\delta_{2,k} + \gamma \norm{b_k})}{t_k}\norm{v_{k+1}-x_*}+\delta_{1,k}^2
\end{align*}
and hence, multiplying by $t_k^2$,
\begin{align}
\label{eq:2021060516a}
\nonumber 2 \gamma t_k^2 r_{k+1} + \norm{v_{k+1} - x_*}^2 &\leq 2 \gamma t_k (t_k  - 1) r_k + \norm{v_k - x_*}^2\\
&\qquad +2 t_k \big(\delta_{2,k} + \gamma \norm{b_k}\big)\norm{v_{k+1}-x_*}+t_k^2\delta_{1,k}^2.
\end{align}
Now, we set $\mathcal{E}_0(x_*) := \norm{v_0 - x_*}^2$ and, for all $k\geq 1$,
$\mathcal{E}_k(x_*) := 2 \gamma t_{k-1}^2 r_k + \norm{v_k - x_*}^2$. Then, 
using the fact that $t_0=1$, that $t_k (t_k  - 1) \leq t^2_{k-1}$, for $k\geq 1$, and
that $\norm{v_{k+1}-x_*} \leq \sqrt{\mathcal{E}_{k+1}(x_*)}$, we obtain
\begin{equation*}
 \mathcal{E}_{k+1}(x_*) 
\leq  \mathcal{E}_{k}(x_*) 
 + \underbrace{2  t_k  (\delta_{2,k} + \gamma\norm{b_k})}_{=:\lambda_k}\sqrt{\mathcal{E}_{k+1}(x_*)}
 +\underbrace{t_k^2 \delta_{1,k}^2}_{=:\xi_k}.
\end{equation*}
Therefore, we obtained an inequality which is in the form considered in Lemma~\ref{lm:recurrences}.
Hence, since $\mathcal{E}_0(x_*)= \norm{v_0- x_*}^{2} = \norm{x_0- x_*}^{2}$,
we have, for every $k \in \N$,
\begin{equation}
\label{eq:20251110b}
  \mathcal{E}_{k}(x_*) \leq 
\left[ \norm{x_0- x_*} 
+   2 \sum_{i=0}^{k-1} t_i \big(\delta_{2,i} + \gamma\norm{b_i}\big)
  + \sqrt{\sum_{i=0}^{k-1} t_i^2 \delta_{1,i}^2}\,\right]^2\!\!,
\end{equation}
and the inequality in the statement follows.
Now, suppose that
$t_k\to +\infty$, and that $(t_k \delta_k)_{k\in\N}$ and $(t_k \norm{b_k})_{k\in \N}$ are summable.
Then, since $\delta_{j,k}\leq \delta_k$ ($j=1,2$) and $\ell^1(\N)\subset \ell^2(\N)$, $(\lambda_k)_{k \in \N}$ and $(\xi_k)_{k\in \N}$ defined above are summable, so, again 
by Lemma~\ref{lm:recurrences}, we derive that the sequence
 $(\mathcal{E}_{k}(x_*))_{k \in \N}$ converges. 
Also, since $(\mathcal{E}_{k}(x_*))_{k \in \N}$ is bounded, there exists $M>0$ such that
\begin{equation*}
\forall\,k \in \N\colon 2 \gamma t_{k-1}^2 r_{k} + \norm{v_k - x_*}^2=  \mathcal{E}_{k}(x_*) \leq M.
\end{equation*}
Thus, as $t_k\to +\infty$, it follows that $F(x_k)-F_*\to 0$. Moreover, recalling \eqref{eq:20251104b}, we have that, for every $k\in \N$,
\begin{equation*}
\norm{x_{k+1}- x_*} \leq \bigg(1 - \frac{1}{t_{k}}\bigg) \norm{x_k - x_*} + \frac{1}{t_k} \norm{v_{k+1}-x_*}\leq \max\{\norm{x_k - x_*}, \sqrt{M}\big\}.
\end{equation*}
Hence, by induction $\norm{x_{k}- x_*} \leq \sqrt{M}$, for every $k \in \N$, and hence  $(x_k)_{k \in \N}$ is bounded.
Now, in order to prove the convergence of the sequence $(x_k)_{k \in \N}$, we 
 invoke Opial's Lemma (Fact~\ref{lem:20160601}) and the subsequent remark. We first prove that the sequential weak  cluster points of $(x_k)_{k \in \N}$ belong to $S_*$.
Let $(x_{k_n})_{n \in \N}$ be any weakly convergent subsequence with $x_{k_n} \rightharpoonup x_*$.
Since $F$ is weakly lower semicontinuous, we have $F(x_*) \leq \liminf_n F(x_{k_n}) = \lim_k F(x_k) = F_*$
and hence $x_* \in S_*$. Next, for any $x^1_*, x^2_* \in S_*$,
we show that the sequence $(\scalarp{x_k, x^2_*- x^1_*})_{k \in \N}$ is convergent.
Indeed, because  $(\mathcal{E}_k(x^1_*))_{k \in \N}$ and $(\mathcal{E}_k(x^2_*))_{k \in \N}$
 both converge, we have that
\begin{align}
\label{eq:20261002c}
\mathcal{E}_k(x^1_*) - \mathcal{E}_k(x^2_*) &= \norm{v_k - x^1_*}^2 - \norm{v_k - x^2_*}^2 \notag\\
& = 2 \scalarp{v_k, x^2_*- x^1_*} + \norm{x^1_*}^2 \textcolor{red}{-} \norm{x^2_*}^2
\end{align}
is convergent as $k\to +\infty$, hence  $(\scalarp{v_k, x^2_*- x^1_*})_{k \in \N}$ converges.
Finally, recalling \eqref{eq:20251104b} we derive that, for every $k \in \N$,
\begin{equation*}
v_{k+1} =  t_k (x_{k+1}-x_k) + x_k
= x_{k+1} + (t_k-1) (x_{k+1}-x_k)
\end{equation*}
and hence that
\begin{equation*}
\scalarp{v_{k+1}, x^2_*- x^1_*} = \scalarp{x_{k+1}, x^2_*- x^1_*} + (t_k -1) \big(\scalarp{x_{k+1}, x^2_*- x^1_*} - \scalarp{x_k, x^2_*- x^1_*}\big).
\end{equation*}
Since $t_k\to +\infty$, we have that $t_k>1$ eventually, so that, thanks to Proposition~\ref{prop:20251112a}\ref{prop:20251112a_0}, 
$1/(t_k - 1)\geq 1/k$ for every integer $k\geq \nu$, for some $\nu\in \N$. Thus, an application of Fact~\ref{lem:radu_enis}, with $a_k = \scalarp{x_{k+\nu}, x^2_*- x^1_*}$ and $\lambda_k = t_{k+\nu}-1$, shows that  $(\scalarp{x_{k}, x^2_*- x^1_*})_{k \in \N}$ converges. The weak convergence of $(x_k)_{k \in \N}$ to a point $x_*\in S_*$ follows from Opial's Lemma (Fact~\ref{lem:20160601}) and Remark~\ref{rmk:opial}, while, the weak convergence of $(y_k)_{k \in \N}$ to $x_*$ follows from \eqref{eq:20190313b}, together with the facts that $1/t_k\to 0$ and $(v_k)_{k\in \N}$ is bounded.
\end{proof}

\begin{remark}\label{rmk:errors}\ 
\begin{enumerate}[label={\rm (\roman*)}]
\item Depending on the asymptotic behavior of the sequence $(t_k)_{k\in\N}$, the summability conditions
on the errors may be more or less demanding: in particular, if $t_k$ is defined by any of the two choices 
in \eqref{eq:20251112d}
with $\alpha\in \left]0,1\right]$ (recall Proposition~\ref{prop:20251112a}\ref{prop:20251112a_ii}-\ref{prop:20251112a_iii}), then they become
\begin{equation}
\label{eq:20251111a}
\sum_{k=0}^{+\infty} k^\alpha \delta_k<+\infty\quad \text{and}\quad \sum_{k=0}^{+\infty} k^\alpha \norm{b_k}<+\infty,
\end{equation}
which are ensured if both $\delta_k$ and $\norm{b_k}$ are  $O(1/k^p)$ with $p>1+\alpha$; and in such case $1-\beta_k \asymp 1/k^\alpha$ and $F(x_k)- F_* = O(1/k^{2\alpha})$. We stress that in all the spectrum of values $\alpha \in \left]0,1\right]$ the iterates of Algorithm~\ref{fista} are ensured to weakly converge to a solution of problem \eqref{eq:mainprob}.
Moreover, concerning the speed of convergence in objective values the same remarks given in \cite{AujDos}
fully apply here too. 
Thus, we have full acceleration when $\alpha=1$, but this corresponds to the strongest assumptions on the error decay. Conversely,  if we slowdown the acceleration by choosing $\alpha\in \left]0,1\right[$, the requirements on the error decay become less stringent. Finally, we note that if $\alpha=0$, $(t_k)_{k\in \N}\equiv 1$, Algorithm~\ref{fista} turns to the  (nonaccelerated) proximal gradient algorithm and the conditions \eqref{eq:20251111a} just state the summability of the error sequences $(\delta_k)_{k\in \N}$ and $(\norm{b_k})_{k\in\N}$. This is in full agreement with well-known facts about the inexact proximal gradient algorithm  (see e.g.~\cite{ComSal}).
\item The condition $t_k\to +\infty$ is ensured if the $t_k$'s satisfy \eqref{eq:20251112b}
and in particular if they are defined according to any of the two choices in \eqref{eq:20251107d}.
\end{enumerate}
\end{remark}

\begin{remark}
Theorem~\ref{mainthm} guarantees that if $t_k$ is defined by any of the two choices 
in \eqref{eq:20251112d} with $\alpha=1$ (the critical regime),
we have $t_k \geq (k+2)/2$ and hence 
 \begin{equation*}
 F(x_k) - F_*  \leq  \frac{2}{\gamma (k+1)^2} 
 \left[ \norm{x_0- x_*} + 
 2 \sum_{i=0}^{k-1} t_i \big(\delta_{2,i} + \gamma\norm{b_i}\big)
 +\sqrt{\sum_{i=0}^{k-1} t_i^2 \delta_{1,i}^2}\,\right]^2.
 \end{equation*}
Moreover,  both $(x_k)_{k\in \N}$ and $(y_k)_{k\in \N}$ weakly converge to an $x_*\in S_*$.
\end{remark}

\begin{remark}
A weaker notion of inexactness in the computation of the proximity operator was proposed in \cite{Villa-Salzo}.
In this case,  $e_k=0$ (so   $\delta_{2,k}=0$ and $\delta_{1,k} = \delta_k$). Then the term $\sum_{i=0}^{k-1} t_i \delta_{2,i}$ in \eqref{eq:20251110b} is zero
and 
the conclusion of Theorem~\ref{mainthm} remains true under the following weaker condition on the errors related to the proximity operator
\begin{equation*}
\sum_{k=0}^{+\infty} t_k^2 \delta^2_k<+\infty.
\end{equation*}
For instance if $t_k\asymp k^{\alpha}$, with $\alpha \in \left]0,1\right]$ and
$\delta_k = O(1/k^p)$, then the condition above is ensured if $p>1/2+\alpha$ (instead of the previous $p>1+\alpha$). 
\end{remark}

\begin{remark}[Comparison with Aujol--Dossal \cite{AujDos}]
\label{rmk:comparisonAD}
The basic perturbed energy estimate and the discrete
Bihari--LaSalle argument used above are closely related to
the argument used in deriving the convergence rate in objective values in \cite{AujDos}. The main difference concerns the
passage from these estimates to weak convergence of the
iterates.
For the parameter family
\begin{equation*}
t_k=\left(\frac{k+a}{a}\right)^d,
\qquad
0<d\leq1,
\qquad
a>\max\{1,(2d)^{1/d}\},
\end{equation*}
Aujol--Dossal establish
\begin{equation*}
\sum_{k=1}^{\infty}k^d(F(x_k)-F_*)<\infty
\quad\text{and}\quad
\sum_{k=1}^{\infty}k^d\|x_k-x_{k-1}\|^2<\infty
\end{equation*}
\cite[Theorem 3.5 and Corollary 3.8]{AujDos}.
Together with bounds on products of momentum coefficients
\cite[Lemma A.5]{AujDos}, these estimates yield convergence
of $\|x_k-x_*\|^2$ for every minimizer $x_*$, allowing an
application of Opial's lemma.
Our proof instead establishes convergence of
$\mathcal E_k(x_*)$ for every minimizer $x_*$ through
Lemma~\ref{lm:recurrences}. Following the energy-difference
argument of \cite{JanRyu}, we then use \eqref{eq:20261002c}
and Fact~\ref{lem:radu_enis} to obtain convergence of
$(\scalarp{x_k,x_*^2-x_*^1})_{k\in\N}$ for every pair of minimizers
$x_*^1,x_*^2$. This avoids the preceding weighted summability
estimates and the bounds on products of momentum
coefficients.
In particular, the argument only requires the parameters $(t_k)_{k\in \N}$
to satisfy the conditions stated in Algorithm~\ref{fista}.
It therefore covers the above power-type acceleration parameters,
as well as the critical choice $t_k=(k+2)/2$
and the original FISTA recursion. 
Finally, coercivity is assumed in \cite{AujDos} to ensure
existence of minimizers and boundedness of the iterates.
Here existence is assumed directly, while boundedness
follows from the energy estimate and the equation $x_{k+1} = (1 - t_k^{-1}) x_k + t_k^{-1} v_{k+1}$.
The same argument also permits removal of coercivity
from their boundedness argument.
\end{remark}

\begin{remark}[Stability of FISTA iterates under relative errors]\label{rmk:relproxerror}
In \cite{Bello}, a perturbed version of FISTA is studied in which
the proximity operators are computed with relative errors.
After rescaling and reindexing to match our notation, Rule~1
and Algorithm~2 in \cite{Bello} can be formulated as follows.
Given $\eta\in\left[0,1-\gamma L\right]$, choose
$x_{k+1},e_k\in\HH$ and $\delta_{1,k}\in\R_+$ such that
\begin{equation}
\label{eq:relproxerror}
\begin{cases}
\dfrac{y_k + e_k - x_{k+1}}{\gamma} - \nabla f(y_k)\in \partial_{\delta^2_{1,k}/(2\gamma)}g(x_{k+1}),\\[2ex]
\delta_{1,k}^2+\norm{e_k}^2 \leq (1 - \gamma L - \eta)\norm{x_{k+1}- y_k}^2
\end{cases}
\end{equation}
and set
\begin{equation}
\label{eq:20261009d}
y_{k+1} = x_{k+1} + \frac{t_k - 1}{t_{k+1}}(x_{k+1}- x_k) - \frac{t_k}{t_{k+1}}e_k.
\end{equation}
The proof of Theorem~\ref{mainthm} can be adapted to this setting
to yield the convergence rate
\begin{equation}
\label{eq:relproxrate}
F(x_k)- F_* \leq \frac{\norm{x_0 - x_*}^2}{2 \gamma t_{k-1}^2},
\end{equation}
which matches the one given in \cite[Theorem~2]{Bello}.
Moreover, if $t_k\to +\infty$ and either $e_k =0$, for every $k\in \N$, or $\eta>0$, then there exists $x_*\in S_*$, such that
$x_k \rightharpoonup x_*$ and $y_k \rightharpoonup x_*$. 
Thus, the stability of FISTA in the critical regime extends
to relative computational errors.
In contrast, \cite{Bello} does not address convergence of the
iterates in any regime.
The proof of these statements is deferred to
Appendix~\ref{sec:app}.
\end{remark}

We now address the convergence of Algorithm~\ref{fistastoch}. We start with a result concerning Assumption~\ref{ass:2}.

\begin{lemma}
\label{lem:ass2}
Let Assumption~\ref{ass:2} hold. Let  $(\delta, \gamma)\in \Delta\times \Gamma$, $y \in \HH$ and set
\begin{equation*}
\Err = \hat{u}(y, \zeta)- \nabla f(y),\ \ 
D = \widehat{G}_\gamma(y,\zeta) - G_\gamma(y),\ \ 
\widehat{G}^\delta_\gamma(y,\zeta) = \frac{y - \prox^\delta_{\gamma g}(y - \gamma\hat{u}(y, \zeta))}{\gamma}.
\end{equation*}
Then the following hold\\[-2ex]
 \begin{enumerate}[label={\rm (\roman*)}]
 \setlength{\itemsep}{3pt} 
 \item\label{lem:ass2_i} $\EE[\norm{D}^2] \leq \EE[\scalarp{D, \Err}] = \EE[\scalarp{\widehat{G}_\gamma(y,\zeta), \Err}] \leq \EE[\norm{\Err}^2]$.
 \item\label{lem:ass2_ib} If $y \in \argmin (f+g)$,  then $\EE[\norm{\widehat{G}_\gamma(y,\zeta)}^2] \leq \rho \sigma^2$. 
 \item\label{lem:ass2_ii} Setting $\vartheta_\rho = 1 - \rho^{-1}$, $A_{\rho, \tau} = 2 \vartheta_\rho + \tau$, and $B_{\rho, \tau} = \vartheta_\rho + \tau$, we have
 \begin{multline*}
\EE[\scalarp{\widehat{G}^\delta_\gamma(y, \zeta), \Err}] \leq \vartheta_{\rho} \EE[\lVert \widehat{G}^\delta_\gamma(y, \zeta)\rVert^2] \\+ \frac{A_{\rho,\tau} \delta}{\gamma}\sqrt{\EE[\lVert \widehat{G}^\delta_\gamma(y, \zeta)\rVert^2]} + \frac{B_{\rho, \tau} \delta^2}{\gamma^2} + \sigma^2 + \frac{\chi \delta}{\gamma}.
\end{multline*}
 \end{enumerate}
\end{lemma}
\begin{proof}
\ref{lem:ass2_i}:
Using the firm nonexpansiveness of the proximity operator we have
\begin{align*}
\norm{D}^2 &= \frac{1}{\gamma^2} \lVert \prox_{\gamma g}(y - \gamma \nabla f(y)) - \prox_{\gamma g}(y - \gamma \hat{u}(y,\zeta)) \rVert^2\\
& \leq \frac{1}{\gamma^2}\scalarp{\prox_{\gamma g}(y - \gamma \nabla f(y)) - \prox_{\gamma g}(y - \gamma \hat{u}(y,\zeta)), \gamma(\hat{u}(y,\zeta) - \nabla f(y))}\\
& = \scalarp{D, \Err}.
\end{align*}
Thus, the first inequality follows by taking the expectation. The equality follows by the unbiasedness of $\hat{u}$. 
As for the last inequality, it is trivial if $\EE[\scalarp{D, \Err}]=0$. Otherwise, by applying Cauchy-Schwarz inequality in $\HH$ and in $L^2(\Omega, \PP)$ and using the first inequality we have
\begin{equation*}
\EE[\scalarp{D, \Err}] \leq \EE[\norm{D}\norm{\Err}] 
\leq \sqrt{\EE[\norm{D}^2]} \sqrt{\EE[\norm{\Err}^2]} \leq \sqrt{\EE[\scalarp{D, \Err}]} \sqrt{\EE[\norm{\Err}^2]}. 
\end{equation*}
Squaring the inequality and dividing by $\EE[\scalarp{D, \Err}]$ the statement follows.

\ref{lem:ass2_ib}: Let $y\in \argmin (f+g)$. Then, since $G_\gamma(y)=0$, we have $D = \widehat{G}_\gamma(y,\zeta)$, and by item \ref{lem:ass2_i} and Assumption~\ref{ass:2}\ref{ass:2_v}, 
\begin{align*}
\EE\big[\lVert\widehat{G}_\gamma(y,\zeta)\rVert^2\big] &= \EE\big[\lVert D\rVert^2\big]\leq \EE\big[\scalarp{\widehat{G}_\gamma(y,\zeta), \Err}\big]\leq \bigg(1 - \frac{1}{\rho}\bigg) \EE\big[\Vert\widehat{G}_\gamma(y,\zeta)\rVert^2\big] + \sigma^2.
\end{align*}
This yields $\EE\big[\Vert\widehat{G}_\gamma(y,\zeta)\rVert^2\big] \leq \rho \sigma^2$
and the statement follows.

\ref{lem:ass2_ii}: Set $d = \prox^\delta_{\gamma g}(y - \gamma \hat{u}(y,\zeta)) - \prox_{\gamma g}(y - \gamma \hat{u}(y,\zeta))$.
We have
\begin{equation}
\label{eq:20261007a}
\widehat{G}^\delta_\gamma(y,\zeta) = \frac{y - \prox_{\gamma g}(y - \gamma \hat{u}(y,\zeta))}{\gamma} - \frac{d}{\gamma}
= \widehat{G}_\gamma(y,\zeta) - \frac{d}{\gamma},
\end{equation}
where, thanks to Proposition~\ref{prop:20251109a}\ref{prop:20251109a_i},  $\norm{d}\leq \delta$ a.s.
Then, by Assumption~\ref{ass:2}\ref{ass:2_v}-\ref{ass:2_vi},
\begin{align*}
\EE[\scalarp{\widehat{G}^\delta_\gamma(y,\zeta),\Err}] &= \EE[\scalarp{\widehat{G}_\gamma(y,\zeta),\Err}] - \gamma^{-1}\EE[\scalarp{d, \Err}]\\
& \leq \vartheta_\rho \EE[\lVert\widehat{G}_\gamma(y,\zeta)\rVert^2] + \sigma^2 + \gamma^{-1}\delta \EE[\norm{\Err}]\\
&\leq \vartheta_\rho \EE[\lVert\widehat{G}_\gamma(y,\zeta)\rVert^2] + \sigma^2 + \gamma^{-1}\delta \tau \sqrt{\EE[\lVert\widehat{G}_\gamma(y,\zeta)\rVert^2 ]} + \gamma^{-1}\delta\chi.
\end{align*}
Next, \eqref{eq:20261007a} and the triangle inequality in $L^2(\Omega, \PP; \HH)$ yields
\begin{equation*}
\sqrt{\EE[\lVert\widehat{G}_\gamma(y,\zeta)\rVert^2 ]} = \sqrt{\EE[\lVert\widehat{G}^\delta_\gamma(y,\zeta) + \gamma^{-1} d\rVert^2 ]}
\leq \sqrt{\EE[\lVert\widehat{G}^\delta_\gamma(y,\zeta)\rVert^2]} + \gamma^{-1}\delta.
\end{equation*}
Therefore, setting, for the sake of brevity, $a = \sqrt{\EE[\lVert\widehat{G}^\delta_\gamma(y,\zeta)\rVert^2 ]}$ and combining the above two inequalities, we have
\begin{align*}
\EE[\scalarp{\widehat{G}^\delta_\gamma(y,\zeta),\Err}] &\leq \vartheta_\rho \bigg(a + \frac{\delta}{\gamma} \bigg)^2 + \frac{\tau \delta}{\gamma}\bigg(a + \frac{\delta}{\gamma} \bigg) + \sigma^2 + \frac{\chi \delta}{\gamma}\\
& = \vartheta_\rho a^2 + \frac{\delta}{\gamma}(2\vartheta_\rho  + \tau ) a + \frac{\delta^2}{\gamma^2}(\vartheta_\rho + \tau)
+ \sigma^2 + \frac{\chi \delta}{\gamma}
\end{align*}
and the statement follows.
\end{proof}

The next theorem provides the second main result of the paper and  includes both a proximal strong growth condition and additive
noise.

\begin{theorem}
\label{thm:composite_fista_general}
Suppose that Assumptions~\ref{ass:1} and \ref{ass:2} hold with $\rho\geq1$ and
$\sigma, \tau, \chi \in \R_+$, $\Delta\subset \R_+$ nonempty set, and $\Gamma = \left]0, 1/(\rho L)\right]$. 
Let $(x_k)_{k \in \N}$ be generated by Algorithm~\ref{fistastoch}, with 
$\gamma_k\leq 1/(\rho L)$ and $\delta_k \in \Delta$, for every $k \in \N$.
Define, for every $k\in \N$,
\begin{align}
c_{1,k}
&:= 2 \left( 6 - \frac{5}{\rho} + 2\tau \right) t_k\delta_k,\notag\\
c_{2,k}
&:= \frac{1}{\rho}\left[
t^2_k \delta_k^2\bigg(3 - \frac{2}{\rho} + 2\tau \bigg) +2\gamma^2_k t^2_k\sigma^2
+2 \gamma_k t^2_k\delta_k\chi\right].
\label{eq:unified_error_sums}
\end{align}
Then, for every $x_*\in S_*$ and every integer $k\geq 1$,
\begin{equation}
\label{eq:inexact_rate_general}
\EE\big[F(x_{k})-F_*\big]\leq
\frac{\rho}{2\gamma_{k-1} t_{k-1}^2}
\left[
\norm{x_0-x_*}+ \sum_{i=0}^{k-1}c_{1,i} + \sqrt{\sum_{i=0}^{k-1}c_{2,i}}\,\right]^2.
\end{equation}
Moreover, suppose that 
\begin{equation}
\label{eq:fista_iterconv}
\gamma_k t^2_k\to +\infty,\qquad \sum_{k=0}^{+\infty}  t_k \delta_k<+\infty,\quad\text{and}\quad
\begin{cases}
\displaystyle\sum_{k=0}^{+\infty} \gamma_k^2 t_k^2 <+\infty &\text{if}\ \sigma>0\\[3ex]
\displaystyle\sum_{k=0}^{+\infty} \chi t_k^2 \delta_k <+\infty &\text{if}\ \sigma=0.
\end{cases}
\end{equation}
Then there exists an $S_*$-valued random vector $\bar x$ such that
$x_k\rightharpoonup\bar x$ and  $y_k\rightharpoonup\bar x$, a.s.
\end{theorem}
\begin{proof}
We define the filtration $\mathfrak{F}_k = \sigma(\zeta_0, \dots, \zeta_{k-1})$ for $k\geq 1$ and $\mathfrak{F}_0$ be the trivial $\sigma$-algebra.
Thus, by definition $y_k$ is $\mathfrak{F}_k$-measurable. 
Let  $k \in \N$ and define 
\begin{equation}
\label{eq:stoch_pg_mapping}
\widetilde{G}_k:=\frac{1}{\gamma_k}(y_k-x_{k+1}),
\qquad
\hat{b}_k:=\hat{u}(y_k,\zeta_k)-\nabla f(y_k),
\end{equation}
so that, as before, $\hat{u}(y_k,\zeta_k) = \nabla f(y_k) + \hat{b}_k$, where the error $\hat{b}_k$ is now random.
We need to rewrite Algorithm~\ref{fistastoch} in a three-sequence form, introducing an auxiliary variable $v_k$.
It follows from the definition of $y_{k+1}$ that, for every $k \in \N$,
\begin{align*}
y_{k+1} &= \bigg(1 - \frac{1}{t_{k+1}} \bigg) x_{k+1} + \frac{t_k}{t_{k+1}} (x_{k+1}-x_k) + \frac{1}{t_{k+1}} x_k + \bigg(\frac 1 \rho - 1\bigg)\frac{t_k}{t_{k+1}} (x_{k+1}- y_k)\\
& = \bigg(1 - \frac{1}{t_{k+1}} \bigg) x_{k+1} + \frac{1}{t_{k+1}} \big[ x_k + t_k(x_{k+1}-x_k) + (\rho^{-1}-1)t_k(x_{k+1}-y_k)\big]\\
& = \bigg(1 - \frac{1}{t_{k+1}} \bigg) x_{k+1} + \frac{1}{t_{k+1}} \big[ \underbrace{x_k + \rho^{-1}t_k (x_{k+1}-y_k) + t_k(y_{k}-x_k)}_{=:v_{k+1}}\big].
\end{align*}
Therefore, for every $k \in \N$,
\begin{equation}
\label{eq:20261008b}
y_k = \bigg(1- \frac{1}{t_k}\bigg) x_k+ \frac{1}{t_k} v_k\qquad(v_0 = y_0).
\end{equation}
Let $k \in \N$. By equation \eqref{eq:20261008b} we have $t_k(y_k - x_k) = v_k - x_k$, hence, 
the definition of $v_{k+1}$ yields 
\begin{equation}
\label{20261009b}
v_{k+1} = x_k + \rho^{-1}t_k (x_{k+1}-y_k) + t_k(y_{k}-x_k) = v_k + \rho^{-1}t_k (x_{k+1}-y_k).
\end{equation}
Then, recalling \eqref{eq:stoch_pg_mapping} and setting $\eta_k = \gamma_k t_k/\rho$, we have 
\begin{equation}
\label{20260929a}
v_{k+1}=v_k- \eta_k\widetilde{G}_k\quad\text{and}\quad x_{k+1} = y_k - \gamma_k \widetilde{G}_k.
\end{equation}
By Lemma~\ref{lem:20190313d}, with $x = y_k$ and $\gamma= \gamma_k$ we obtain, for every $z\in \HH$,
\begin{align}
\quad2\gamma_k\big(F(x_{k+1})-F(z)\big)
\leq{}&\norm{y_k-z}^2-\norm{x_{k+1}-z}^2
-(1-\gamma_k L)\norm{y_k-x_{k+1}}^2\notag\\[1ex]
&-2\gamma_k\scalarp{x_{k+1}-z,\hat{b}_k}
+2\scalarp{x_{k+1}-z,e_k}+\delta_{1,k}^2,
\label{eq:20190313e2}
\end{align}
where $\delta_{1,k}, \delta_{2,k}, \delta_{k}\textcolor{red}{\,\geq\,}0$ and $e_k\in \HH$ are such that
$\delta_{1,k}^2+ \delta_{2,k}^2 \leq \delta_{k}^2$ and $\norm{e_k}\leq \delta_{2,k}$.
Now let $x_*\in S_*$ and choose
\begin{equation}
\label{eq:composite_zk}
z:=\left(1-\frac1{t_k}\right)x_k+\frac1{t_k}x_*.
\end{equation}
We note that by equation \eqref{eq:20261008b} we obtain
\begin{equation}
\label{eq:20261008c}
y_k - z = \frac{1}{t_k}(v_k - x_*),
\end{equation}
and, using \eqref{20260929a} and \eqref{eq:20261008c},
\begin{equation}
\label{eq:20261012b}
x_{k+1}-z 
=\frac{v_k-x_*}{t_k}+\frac{\gamma_k}{\eta_k}(v_{k+1}-v_{k})
=\frac{1}{t_k}\big[(1-\rho)(v_k-x_*) +\rho(v_{k+1}-x_*)\big].
\end{equation}
Hence, it follows that\footnote{We use the identity $\norm{(1-\lambda) u + \lambda w}^2 = (1-\lambda)\norm{u}^2 + \lambda \norm{w}^2 - (1-\lambda)\lambda \norm{u-w}^2$, which is valid for every $u,w\in \HH$ and $\lambda \in \R$.}
\begin{align*}
\label{eq:20260929b}
&\norm{y_k - z}^2 - \norm{x_{k+1}- z}^2\notag\\[1ex]
 &\quad= \frac{1}{t_k^2}\big[\norm{v_k - x_*}^2 -(1-\rho)\norm{v_k - x_*}^2 - \rho \norm{v_{k+1}-x_*}^2 + (1-\rho)\rho \norm{v_{k+1}-v_k}^2 \big]\notag\\
 &\quad =\frac{\rho}{t^2_k}\big(\norm{v_k - x_*}^2-\norm{v_{k+1} - x_*}^2\big) + \frac{(1-\rho)\rho \eta_k^2}{t^2_k } \lVert \widetilde{G}_k\rVert^2\notag\\
 &\quad =\frac{\rho}{t^2_k}\big(\norm{v_k - x_*}^2-\norm{v_{k+1} - x_*}^2\big) + \gamma^2_k\bigg(\frac{1}{\rho} - 1\bigg) \lVert \widetilde{G}_k\rVert^2.
\end{align*}
Moreover, setting $r_k:=F(x_k)-F_*$,
and recalling \eqref{eq:composite_zk}, by convexity,
\begin{equation}
\label{eq:composite_zk_convexity}
F(z)-F_*
\leq\left(1-\frac1{t_k}\right)r_k.
\end{equation}
Therefore, using the above identity for $\norm{y_k - z}^2 - \norm{x_{k+1}- z}^2$ and \eqref{eq:composite_zk_convexity} in \eqref{eq:20190313e2}, together with the fact that $y_k - x_{k+1}= \gamma_k \tilde{G}_k$, we obtain
\begin{align*}
2\gamma_k r_{k+1} &\leq 2\gamma_k \bigg(1\! -\! \frac{1}{t_k}\bigg) r_k
\!+\! \frac{\rho}{t^2_k}\big(\norm{v_k - x_*}^2\!\!-\norm{v_{k+1} - x_*}^2\big)\! + \!\gamma^2_k\bigg(\frac{1}{\rho} + \!\gamma_k L- 2\bigg) \lVert\widetilde{G}_k\rVert^2\\[1ex]
&\qquad -2\gamma_k\scalarp{x_{k+1}-z,\hat{b}_k}
+2\scalarp{x_{k+1}-z,e_k}+\delta_{1,k}^2,
\end{align*}
which multiplied by $t_k^2/\rho$ and using that $\gamma_{k}\leq \gamma_{k-1}$ and $\delta_{1,k} \leq \delta_k$ yields
\begin{align}
\label{eq:composite_energy_pathwise}
2 \frac{t_k^2\gamma_{k}}{\rho} r_{k+1} + \norm{v_{k+1} - x_*}^2 &\leq 2 \frac{t_k (t_k - 1)\gamma_{k-1}}{\rho} r_k + \norm{v_k - x_*}^2\notag\\[1ex]
&\quad+ \frac{\gamma_k^2 t_k^2}{\rho} \bigg(\frac{1}{\rho} + \gamma_k L- 2\bigg)\lVert\widetilde{G}_k\rVert^2 -2\frac{\gamma_k t^2_k}{\rho} \scalarp{x_{k+1}-z,\hat{b}_k}\notag\\[1ex]
&\quad 
+2\frac{t^2_k}{\rho}\scalarp{x_{k+1}-z,e_k}+ \frac{t_k^2 \delta_{k}^2}{\rho}.
\end{align}
Now, from $x_{k+1}-z=(v_k - x_*)/t_k-\gamma_k\widetilde G_k$ and the fact that $\EE[\hat{b}_k\,\vert\, \mathfrak{F}_k]=0$, we have
\begin{equation*}
-\EE\big[\scalarp{x_{k+1}-z,\hat{b}_k}\,\vert\, \mathfrak{F}_k\big]
=\gamma_k\EE\big[\scalarp{\widetilde G_k,\hat{b}_k}\,\vert\,\mathfrak{F}_k\big].
\end{equation*}
Moreover, by Lemma~\ref{lem:ass2}\ref{lem:ass2_ii}, applied conditionally on $\mathfrak F_k$, we have
\begin{equation*}
\EE\big[\scalarp{\widetilde G_k,\hat{b}_k}\,\vert\,\mathfrak{F}_k\big] \leq 
\bigg(1 - \frac{1}{\rho}\bigg) \EE[\lVert\widetilde{G}_k\rVert^2\,\vert\,\mathfrak{F}_k]
+\frac{\delta_k}{\gamma_k}A_{\rho,\tau}\sqrt{\EE[\lVert\widetilde{G}_k\rVert^2\,\vert\,\mathfrak{F}_k]}
+\frac{\delta_k^2}{\gamma_k^2}B_{\rho,\tau} +\sigma^2+\frac{\chi\delta_k}{\gamma_k}.
\end{equation*}
Hence, 
for the sum of the third and fourth term in the RHS of \eqref{eq:composite_energy_pathwise} we have
\begin{align}
\label{eq:no_kappa_noise_bound}
\EE\bigg[\frac{\gamma_k^2 t_k^2}{\rho} &\bigg( \frac{1}{\rho} +\gamma_k L -  2\bigg)\lVert\widetilde{G}_k\rVert^2 -2\frac{\gamma_k t^2_k}{\rho} \scalarp{x_{k+1}-z,\hat{b}_k}\,\Big\vert\,\mathfrak{F}_k\bigg]\notag\\[1ex]
&\qquad\leq \frac{\gamma_k^2 t_k^2}{\rho}\left( \gamma_k L - \frac{1}{\rho}\right) \EE[\lVert\widetilde{G}_k\rVert^2\,\vert\,\mathfrak{F}_k]
+2 \frac{\gamma_k t^2_k \delta_k}{\rho}A_{\rho,\tau}\sqrt{\EE[\lVert\widetilde{G}_k\rVert^2\,\vert\,\mathfrak{F}_k]}\notag\\
&\qquad\qquad+2B_{\rho,\tau}\frac{t^2_k\delta_k^2}{\rho}+2\frac{\gamma^2_k t^2_k}{\rho}\sigma^2+2\frac{\gamma_k t^2_k\delta_k}{\rho}\chi.
\end{align}
Now, since $\gamma_k \leq 1/(\rho L)$ the first term on the right hand side of the above inequality can be removed.
As for the inexact-proximal term in \eqref{eq:composite_energy_pathwise}, 
using \eqref{eq:20261012b} and the fact that $\norm{e_k}\leq\delta_{2,k}\leq \delta_k$, pointwise, we have
\begin{align}
2\frac{t^2_k}{\rho}
\EE\big[&\scalarp{x_{k+1}-z,e_k}\,\vert\,\mathfrak{F}_k\big]\notag\\
&\leq2t_k \left|\frac{1}{\rho} -1\right|
\delta_{k}\norm{v_k-x_*}
+2t_k\delta_{k}
\EE\big[\norm{v_{k+1}-x_*}\,\vert\,\mathfrak{F}_k\big].
\label{eq:inexact_e_bound}
\end{align}
Now, set $\mathcal E_0(x_*) = \norm{v_0 - x_*}^2$ and 
$\mathcal E_k(x_*):=2  \rho^{-1}\gamma_{k-1}t^2_{k-1} r_k+\norm{v_k-x_*}^2$, for every $k\geq 1$.
Combining \eqref{eq:composite_energy_pathwise} with
\eqref{eq:no_kappa_noise_bound} and \eqref{eq:inexact_e_bound},
and the fact that $t_k(t_k - 1)\leq t^2_{k-1}$ and  $\gamma_k \leq 1/(\rho L)$, we have
\begin{align}
\label{eq:20261009a}
\EE\big[\mathcal E_{k+1}(x_*)\,\vert\,\mathfrak{F}_k\big] &\leq\mathcal E_k(x_*)
 + 2\frac{\gamma_k t^2_k \delta_k}{\rho} A_{\rho,\tau}\sqrt{\EE[\lVert\widetilde{G}_k\rVert^2\,\vert\,\mathfrak{F}_k]}\notag\\
&+2t_k \bigg( 1 - \frac{1}{\rho}\bigg)\delta_{2,k}\norm{v_k - x_*}
 +2t_k\delta_{k}\sqrt{\EE\big[\norm{v_{k+1}-x_*}^2\,\vert\,\mathfrak{F}_k\big]}\notag\\
&+\frac{t^2_k}{\rho}
       (\delta_{k}^2+2B_{\rho,\tau}\delta_k^2)
 +2\frac{\gamma^2_k t^2_k}{\rho}\sigma^2+2\frac{\gamma_k t^2_k\delta_k}{\rho}\chi.
\end{align}
Next, since $\eta_k\widetilde G_k=v_k-v_{k+1}$, we have $\eta_k\lVert\widetilde G_k\rVert\leq\norm{v_k-x_*}+\norm{v_{k+1}-x_*}$,
and hence, Minkowski's inequality gives
\begin{equation}
\label{eq:X_by_p}
\frac{\gamma_k t_k}{\rho}\sqrt{\EE[\lVert \widetilde{G}_k\rVert^2\,\vert\,\mathfrak{F}_k]}
\leq\norm{v_k-x_*}+\sqrt{\EE\big[\norm{v_{k+1}-x_*}^2\,\vert\,\mathfrak{F}_k\big]}.
\end{equation}
Finally, since $\norm{v_k -x_*}^2\leq\mathcal E_k(x_*)$ and
$\EE\big[\norm{v_{k+1}-x_*}^2\,\vert\,\mathfrak{F}_k\big]\leq
\EE\big[\mathcal E_{k+1}(x_*)\,\vert\,\mathfrak{F}_k\big]$, we derive from \eqref{eq:20261009a}
\begin{equation}
\label{eq:inexact_explicit_recursion}
\EE\big[\mathcal E_{k+1}(x_*)\,\vert\,\mathfrak{F}_k\big]
\leq\mathcal E_k(x_*)+a_k\sqrt{\mathcal E_k(x_*)}
+b_k\sqrt{\EE\big[\mathcal E_{k+1}(x_*)\,\vert\,\mathfrak{F}_k\big]}+\xi_k,
\end{equation}
where
\begin{align*}
a_k&=
2 t_k \delta_k \big( 3(1 - \rho^{-1}) + \tau\big),\notag\\
b_k&=
2 t_k \delta_k \big(2(1 - \rho^{-1}) + \tau + 1\big),\notag\\
\xi_k&=\rho^{-1}t^2_k \delta_k^2
       (1+2B_{\rho,\tau})
+2 \rho^{-1}\gamma^2_k t^2_k\sigma^2
+2 \rho^{-1}\gamma_k t^2_k\delta_k\chi.
\end{align*}
Now we note that \eqref{eq:inexact_explicit_recursion} implies
\begin{equation}
\label{eq:inexact_explicit_recursion2}
\EE\big[\mathcal E_{k+1}(x_*)\,\vert\,\mathfrak{F}_k\big]
\leq\mathcal E_k(x_*)+(a_k+b_k)\sqrt{\EE\big[\mathcal E_{k+1}(x_*)\,\vert\,\mathfrak{F}_k\big]}+\xi_k.
\end{equation}
Indeed, on the set  $\Omega_1 =\big\{\omega \in \Omega\,\vert\,\EE\big[\mathcal E_{k+1}(x_*)\,\vert\,\mathfrak{F}_k\big](\omega) \leq \mathcal E_k(x_*)(\omega)\big\}$, \eqref{eq:inexact_explicit_recursion2} clearly holds. On the other hand, on the complement $\Omega_1^c$, 
$\mathcal E_k(x_*) \leq \EE\big[\mathcal E_{k+1}(x_*)\,\vert\,\mathfrak{F}_k\big]$ and \eqref{eq:inexact_explicit_recursion} implies \eqref{eq:inexact_explicit_recursion2}.
Taking total expectations in \eqref{eq:inexact_explicit_recursion2},
and using Jensen's inequality for conditional expectation and the tower rule gives
\begin{equation}
\label{eq:inexact_scalar_recursion}
\EE\big[\mathcal E_{k+1}(x_*)\big]\leq \EE\big[\mathcal E_k(x_*)\big]+ (a_k+b_k)\sqrt{\EE\big[\mathcal E_{k+1}(x_*)\big]}+\xi_k.
\end{equation}
Therefore, by Lemma~\ref{lm:recurrences}, with $\lambda_k = a_k + b_k$, we obtain
\begin{equation}
\label{eq:inexact_energy_explicit}
\EE[\mathcal{E}_{k+1}(x_*)]\leq 
\left[
\norm{x_0 - x_*} + \sum_{i=0}^k(a_i+b_i) +\sqrt{\sum_{i=0}^k\xi_i}\,\right]^2.
\end{equation}
Since $a_k+b_k=c_{1,k}$,
$\xi_k=c_{2,k}$, and
$2\gamma_k t^2_k\rho^{-1}\EE\big[r_{k+1}\big]\leq \EE[\mathcal{E}_{k+1}(x_*)]$,
 \eqref{eq:inexact_rate_general} follows.
 
Concerning the second part of the statement, 
 We now suppose that $\gamma_k t^2_k\to +\infty$ and that $(\delta_k t_k)_{k \in \N}$
is summable. Since $\ell^1(\N)\subset \ell^2(\N)$, we have that $(t^2_k \delta^2_k)_{k \in \N}$
is summable. Moreover, if $\sigma>0$ and  $(\gamma_k t_k)_{k \in \N}$ is square summable, since $\ell^2(\N)\cdot \ell^2(\N)\subset \ell^1(\N)$,
we have that $(\gamma_k^2 t^2_k)_{k \in \N}$ and
$(\gamma_k t^2_k \delta_k)_{k \in \N}$ are summable. Hence,
\begin{equation}
\label{eq:20261010a}
\sum_{k=0}^{+\infty} (a_k + b_k)<+\infty\quad
\text{and}\quad \sum_{k=0}^{+\infty} \xi_k<+\infty.
\end{equation}
On the other hand, if $\sigma=0$, since $(\gamma_k)_{k\in \N}$ is upper bounded,
if $(\chi t_k^2 \delta_k)_{k \in \N}$ is summable, we have that \eqref{eq:20261010a} still holds.
Therefore, recalling \eqref{eq:inexact_scalar_recursion}, Lemma~\ref{lm:recurrences} ensures that
$(\EE\big[\mathcal E_k(x_*)\big])_{k\in \N}$ is convergent and hence bounded.
Define the positive $\mathfrak F_k$-measurable random variable
\begin{equation*}
\varepsilon_k=
(a_k+b_k)\sqrt{\EE\big[\mathcal E_{k+1}(x_*)\,\vert\, \mathfrak{F}_k\big]}.
\end{equation*}
By Jensen's inequality and the tower property,
\begin{equation*}
\EE\bigg[\sum_{k=0}^{+\infty} \varepsilon_k \bigg]=
\sum_{k=0}^{+\infty}\EE\big[\varepsilon_k\big]
\leq\sum_{k=0}^{+\infty}(a_k+b_k)\sqrt{\EE[\mathcal{E}_{k+1}(x_*)]}<+\infty.
\end{equation*}
Hence $\sum_{k=0}^{+\infty}\varepsilon_k<+\infty$ almost surely.
Applying Fact~\ref{thm:RiSi71} to
\begin{equation*}
\EE\big[\mathcal E_{k+1}(x_*)\,\vert\,\mathfrak{F}_k\big]\leq\mathcal E_k(x_*)+\varepsilon_k+\xi_k
\end{equation*}
 now ensures that the sequence $(\mathcal{E}_k(x_*))_{k\in \N}$ converges almost surely.
This also implies that
 there exists a finite measurable function
$M_1\colon \Omega\to \R$ such that
\begin{equation*}
\sup_{k\in\N} \bigg[2 \frac{\gamma_{k-1} t^2_{k-1}}{\rho} (F(x_k)- F_*) + \norm{v_k - x_*}^2 \bigg]=\sup_{k\in \N} \mathcal{E}_k(x_*) \leq M_1\ \text{a.s.},
\end{equation*}
which yields that $(v_k)_{k\in \N}$ is almost surely bounded.
Thus, since $\gamma_{k-1} t^2_{k-1}\to +\infty$, we have $F(x_k)-F_*\to 0$ a.s.
and hence, as before, the sequential weak cluster points of $(x_k)_{k\in \N}$
are contained in $S_*$ a.s.
On the other hand, set
\begin{equation}
w_{k+1}=x_k+t_k(x_{k+1}-x_k).
\end{equation}
Directly from \eqref{20261009b} and \eqref{eq:20261008b}, we obtain
\begin{equation}
\label{eq:vw_general}
v_{k+1}=v_k+\frac{1}{\rho}(w_{k+1}-v_k),
\qquad
w_{k+1}=v_k+\rho(v_{k+1}-v_k).
\end{equation}
Thus, the sequence $(w_{k+1})_{k\in \N}$ is bounded
almost surely. In particular $\norm{w_{k+1}- x_*} \leq M_2:=(2 \rho -1)\sqrt{M_1}$ a.s. Since, by the definition of $w_{k+1}$,
\begin{equation}
\label{eq:20261012c}
x_{k+1}=\left(1-\frac1{t_k}\right)x_k+\frac1{t_k}w_{k+1},
\end{equation}
we have that, for every $k\in \N$,
\begin{equation*}
\norm{x_{k+1}- x_*} \leq \bigg(1 - \frac{1}{t_{k}}\bigg) \norm{x_k - x_*} + \frac{1}{t_k} \norm{w_{k+1}-x_*}\leq \max\{\norm{x_k - x_*}, M_2\big\},
\end{equation*}
so by induction, $\norm{x_{k}- x_*} \leq M_2$ a.s., and hence $(x_k)_{k \in \N}$ is bounded a.s.
Since, for every $x_*^1, x_*^2\in S_*$,
$\mathcal{E}_k(x_*^1) - \mathcal{E}_k(x_*^2) = 2\scalarp{v_k,x_*^2-x_*^1} + \norm{x_*^1}^2-\norm{x_*^2}^2$ is almost surely convergent,
the sequences $(\scalarp{v_k,x_*^2-x_*^1})_{k\in \N}$ and $(\scalarp{w_{k+1},x_*^2-x_*^1})_{k \in \N}$
converge almost surely to the same limit, by \eqref{eq:vw_general}.
Thus, continuing following the same line of arguments as in the proof of Theorem~\ref{mainthm} and relying on
 \eqref{eq:20261012c} we have
 \begin{equation*}
\scalarp{w_{k+1}, x^2_*- x^1_*} = \scalarp{x_{k+1}, x^2_*- x^1_*} + (t_k -1) \big(\scalarp{x_{k+1}, x^2_*- x^1_*} - \scalarp{x_k, x^2_*- x^1_*}\big)
\end{equation*}
and hence once again Fact~\ref{lem:radu_enis} ensures that
$(\scalarp{x_k,x_1^*-x_2^*})_{k\in \N}$ converges almost surely, for every $x_1^*, x_2^*\in S_*$.
The conclusion about the almost-sure weak convergence of $(x_k)_{k\in\N}$ to a $S_*$-valued random vector $x_*$ now follows from the stochastic Opial's Lemma~\ref{StochOpial}, while the corresponding convergence of $(y_k)_{k\in\N}$ to $x_*$ follows again from \eqref{eq:20261008b}, considering that $1/t_k\to 0$ and $(v_k)_{k\in \N}$ is bounded a.s.
\end{proof}

The conditions \eqref{eq:fista_iterconv} in
Theorem~\ref{thm:composite_fista_general}, which ensure
almost-sure weak convergence of the FISTA iterates, require
the stepsizes $\gamma_k$ to tend to zero when $\sigma>0$.
This leads to a slower guaranteed convergence rate.
On the other hand, when $\sigma=0$, constant stepsizes
are allowed, and the accelerated convergence rate can
be preserved.
We examine these two cases in the following Remark and Corollary.

\begin{remark}
As discussed in Remark~\ref{rmk:ass2}\ref{rmk:ass2_i},
the case $\sigma=0$ yields an interpolation property under
the sample-loss interpretation described there.
Set $C_{\rho,\tau} = 2 (6 - 5/\rho + 2\tau)$ and $D_{\rho,\tau}=3- 2/\rho +2\tau$,
and choose $\gamma_k\equiv\gamma:=1/(\rho L)$.
Then the bound in Theorem~\ref{thm:composite_fista_general}
becomes, for every $x_*\in S_*$ and every integer $k\geq1$,
\begin{align*}
\EE\big[F(x_k)-F_*\big]
&\leq
\frac{\rho^2L}{2t_{k-1}^2}
\bigg[
\norm{x_0-x_*}
+C_{\rho,\tau}\!\sum_{i=0}^{k-1}t_i\delta_i
+\sqrt{
\frac{1}{\rho}\sum_{i=0}^{k-1}
\left(
D_{\rho,\tau}t_i^2\delta_i^2
+2\gamma t_i^2\delta_i\chi
\right)}
\bigg]^2.
\end{align*}
Suppose additionally that $t_k\to+\infty$.
If $\chi=0$, summability of $(t_k\delta_k)_{k\in\N}$
ensures almost-sure weak convergence of both iterate
sequences to a common minimizer and the rate
$\EE\big[F(x_k)-F_*\big]=O(t_{k-1}^{-2})$.
If $\chi>0$, the same conclusions hold under the stronger
condition that $(t_k^2\delta_k)_{k\in\N}$ is summable.
In particular, when $t_k\asymp k$, these conditions preserve
the accelerated rate $O(k^{-2})$.
When the proximity operators are computed exactly,
that is, $\delta_k=0$ for every $k\in\N$, and
for either the choice $t_k=(k+2)/2$ or the classical
FISTA recursion, the bound simplifies to
\begin{equation*}
\EE\big[F(x_k)-F_*\big]
\leq
\frac{2\rho^2L}{(k+1)^2}\norm{x_0-x_*}^2.
\end{equation*}
Thus, this result not only extends to the composite setting the accelerated rate established in the smooth setting in \cite[Theorem~2]{Vas19} and \cite[Theorem~3 and eq.~(141)]{Her25}, but also, notably, establishes the weak convergence of the iterates.
\end{remark}

\begin{corollary}[Inexact stochastic FISTA under uniformly bounded variance]
\label{cor:bounded_variance}
Suppose that Assumption~\ref{ass:1} and
Assumption~\ref{ass:2}\ref{ass:2_i}-\ref{ass:2_iv} hold and that 
\begin{equation}
\label{eq:uniform_variance}
(\forall\,y\in \HH)\quad
\EE\big[\norm{\hat u(y,\zeta)-\nabla f(y)}^2\big]\leq\sigma^2,
\end{equation}
where $\sigma\geq0$. 
Let $(x_k)_{k \in \N}$ be generated by Algorithm~\ref{fistastoch}, with $\rho=1$ and
$\gamma_k \leq 1/(\rho L)$, for every $k \in \N$.
Set, for every $k\in \N$,
$c_k = t^2_k \delta_k^2 +2\gamma^2_k t^2_k\sigma^2+2 \gamma_k t^2_k\delta_k\sigma$.
Then,  for every $x_*\in S_*$ and integer $k\geq1$,
\begin{equation*}
\EE\big[F(x_{k})-F_*\big]\leq
\frac{1}{2\gamma_{k-1} t_{k-1}^2}
\left[ \norm{x_0-x_*}+
2 \sum_{i=0}^{k-1}  t_i\delta_i
+\sqrt{\sum_{i=0}^{k-1} c_i}\right]^2.
\end{equation*}
Moreover, suppose that $\gamma_k t^2_k\to +\infty$ and that $(\gamma_k t_k)_{k \in \N}$ is square summable and that $(\delta_k t_k)_{k \in \N}$
is summable. Then there exists an $S_*$-valued random vector $\bar x$ such that
$x_k\rightharpoonup\bar x$ and  $y_k\rightharpoonup\bar x$, a.s.
\end{corollary}
\begin{proof}
It follows from Lemma~\ref{lem:ass2}\ref{lem:ass2_i} (using the same notations) that, 
\begin{equation*}
(\forall\,y\in \HH)\quad
\EE[\scalarp{\widehat{G}_\gamma(y, \zeta), \Err}] = \EE[\scalarp{D, \Err}] \leq \EE[\norm{\Err}^2]\leq \sigma^2.
\end{equation*}
Thus, Assumption~\ref{ass:2}\ref{ass:2_v} holds with $\rho=1$.
Moreover, by Jensen's inequality, Assumption~\ref{ass:2}\ref{ass:2_vi} is satisfied as well with $\tau=0$ and $\chi=\sigma$.
Therefore, the conclusion follows directly from Theorem~\ref{thm:composite_fista_general}.
\end{proof}

\begin{remark}
Assume that, in Theorem~\ref{thm:composite_fista_general}, $\sigma>0$ and
 that $t_k \asymp k^{\alpha}$, with $\alpha \in \left]0,1\right]$. Then the conditions \eqref{eq:fista_iterconv} become
\begin{equation}
\label{eq:20251124a}
\gamma_k k^{2\alpha}\to +\infty,
\quad\sum_{k=0}^{+\infty} \delta_k k^\alpha<+\infty,
\quad\text{and}\quad
\sum_{k=0}^{+\infty} \gamma_k^2 k^{2\alpha}<+\infty.
\end{equation}
Suppose in addition that, for every $k\in \N$,
\begin{equation}
\label{eq:20251126a}
\gamma_k = \frac{\gamma}{k^q(1+\log k)^r}\quad (q>0, r\geq 0)\quad\text{and}
\quad \delta_k = O\bigg(\frac{1}{k^p}\bigg)\quad(p>0).
\end{equation}
Then \eqref{eq:20251124a} translates into the following conditions on the exponents\footnote{The convergence of the series can be easily analyzed by the Cauchy's condensation test \cite{Knopp}.} $p,q, r$
\begin{equation*}
\bigg[\bigg(\alpha+\frac{1}{2}<q<2\alpha\bigg)\quad\text{or}\quad
\bigg(\alpha+\frac{1}{2}= q\quad\text{and}\quad \alpha, r>\frac{1}{2} \bigg)\bigg]
\quad\text{and}\quad p>\alpha+1,
\end{equation*}
and, if they are satisfied, Algorithm~\ref{fistastoch} features the following rate of convergence in expectation for the objective values
\begin{equation}
\label{eq:20251125c}
\EE[F(x_k)]-F_* = O\bigg( \frac{(1+\log k )^r}{k^{2\alpha-q}}\bigg).
\end{equation}
Thus, the best rate is obtained for $\alpha=1$ and $q=3/2$, which yields (with $r>1/2$)
\begin{equation}
\label{eq:20251126b}
\EE[F(x_k)]-F_* = O\bigg( \frac{(1+\log k )^r}{k^{1/2}}\bigg).
\end{equation}
Since for every integer $k\geq 1$,
\begin{equation*}
\sum_{i=1}^{k} \gamma_i^2 i^{2\alpha}= \sum_{i=1}^{k} \frac{\gamma^2}{i(1+\log i)^{2r}} \leq \gamma^2\cdot
\begin{cases}
1 + \log k & \text{if}\ r=0\\[1ex]
1+ \dfrac{(1+\log k)^{1-2 r}-1}{1-2 r}&\text{if}\ 0<r<1/2\\[2ex]
1+\log(1+\log k) &\text{if}\ r=1/2\\[1ex]
1 + \dfrac{1}{2 r-1} &\text{if}\ r>1/2,
\end{cases}
\end{equation*}
optimizing the choice of $r$ too leads to $r=1/2$ and the rate
\begin{equation}
\label{eq:20251126c}
\EE[F(x_k)]-F_* = O\bigg( \frac{(1+\log k )^{1/2} (1+ \log(1+\log k))}{k^{1/2}}\bigg),
\end{equation}
although, in such setting, the iterates are no longer ensured to converge.
We finally observe that
the rates \eqref{eq:20251126b} and \eqref{eq:20251126c} essentially match the analogue rates of the stochastic (nonaccelerated) proximal gradient algorithm under the same assumptions on the noise, this is consistent with the optimal worst-case rate $O(k^{-1/2})$ \cite{Lan}.
\end{remark}


\bibliographystyle{siamplain}

\begin{thebibliography}{99}
\bibitem{ApiAuDos} 
V.~Apidopoulos, J.-F.~Aujol, and  Ch.~Dossal, 
The differential inclusion modeling FISTA algorithm and optimality of convergence rate in the case 
$b\leq 3$, 
{\it SIAM J.~Optim.}, 28 (2018), pp.~551--574.

\bibitem{AraBau}
A. Arakcheev and H.H. Bauschke,
On Opial's Lemma, arXiv, 2026. \url{https://arxiv.org/pdf/2503.22004}


\bibitem{AttouChba}
H.~Attouch, Z.~Chbani, J.~Peypouquet, and P.~Redont,
Fast convergence of inertial dynamics and algorithms with asymptotic 
vanishing viscosity. {\it Math.~Program.~B.}, 168 (2018),  pp.~123--175.


\bibitem{AttouBot}
H.~Attouch, R.~I.~Bo\c{t}, D.~A.~Hulett, and D.-K.~Nguyen,
Recovering Nesterov accelerated dynamics from Heavy Ball dynamics via time rescaling, 
\textit{SIAM J.~Control Optim.}, 64(3), 1040--1067, 2026.


\bibitem{AttouCabo}
H.~Attouch and A.~Cabot,
Convergence rates of inertial forward-backward algorithms.
\textit{SIAM J.~Optim.}, 28 (2018), pp.~849--874.

\bibitem{AttouPeyp}
H.~Attouch and J.~Peypouquet,
The rate of convergence of Nesterov's accelerated forward-backward method is actually faster than $1/k^2$.
\textit{SIAM J.~Optim.}, 26 (2016), pp.~1824--1834.


\bibitem{AujDos} 
J.-F.~Aujol and Ch.~Dossal, 
Stability of over-relaxations for the forward-backward
algorithm, application to FISTA,  
{\it SIAM J.~Optim.}, 25 (2015),  pp.~2408--2433.

\bibitem{Barre}
M.~Barré, A.~B.~Taylor, and F.~Bach,
Principled analyses and design of first-order methods with inexact proximal operators,
{\it  Math.~Program.~B.}, 201 (2023), pp.~185--230.

\bibitem{BauComb_book} 
H.~H.~Bauschke and P.~L.~Combettes,
{\em  Convex Analysis and Monotone Operator Theory in Hilbert spaces}, 2nd ed. Springer, New York,
2017.



\bibitem{BT}
A.~Beck and M.~Teboulle,
A Fast Iterative Shrinkage-Thresholding Algorithm for linear inverse problems,
{\it SIAM J.~Imaging Sci.}, 2 (2009), pp.~183--202.

\bibitem{BKUW23}
F.~Beiser, B.~Keith, S.~Urbainczyk, and B.~Wohlmuth,
Adaptive sampling strategies for risk-averse stochastic optimization
with constraints,
{\it IMA J.~Numer.~Anal.}, 43 (2023), pp.~3729--3765.


\bibitem{Bello}
Y.~Bello-Cruz, M.L.N.~Gon\c{c}alves, and N.~Krislock,
On FISTA with a relative error rule
{\it Comput.~Optim.~Appl.}, 84 (2023), pp.~295--318.



\bibitem{BotChen}
R. I.~Bo\c{t}, E.~Chenchene, E.R.~Csetnek, and D.~A.~Hulett,
Accelerating Diagonal Methods for Bilevel Optimization: Unified Convergence via Continuous-Time Dynamics, 
\emph{Math. Program.}, 2026. \url{https://doi.org/10.1007/s10107-026-02376-8}

\bibitem{BotFad}
R. I. Bo\c{t}, J.~Fadili, and D.-K.~Nguyen,
The iterates of Nesterov's accelerated algorithm converge in the critical regime, arXiv, 2025. \url{https://arxiv.org/abs/2510.22715v2}


\bibitem{ChamDoss}
A. Chambolle and C. Dossal, 
On the convergence of the iterates of the ``Fast Iterative Shrinkage/Thresholding Algorithm''. 
{\it J.~Optim.~Theory Appl.}, 166 (2015), pp.~968--982.

\bibitem{Cohn}
D.~L.~Cohn,
{\it Measure Theory}, 2nd ed. 
Birkh\"auser, New York, 2013.

\bibitem{ComPes}
P.~L.~Combettes and J-C.~Pesquet, 
Stochastic quasi-Fej\'er block-coordinate fixed point iteration with random sweeping. 
{\it SIAM J.~Optim.} 25 (2015), pp.~1221--1248.



\bibitem{ComSal}
P.~L.~Combettes, S.~Salzo and S.~Villa, 
Consistent learning by composite proximal thresholding. 
{\it Math.~Program.~B.}, 167 (2018), pp.~99--127.

\bibitem{Guler}
O. G\"uler,
New proximal point algorithms for convex minimization. 
{\it SIAM J.~Optim.} 2 (1992), pp.~649--664.


\bibitem{Gup24} 
K.~Gupta, J.W.~Siegel, S.~Wojtowytsch,
Nesterov acceleration despite very noisy gradients,  
{\it NeurIPS}, 2024.


\bibitem{JanRyu}
U. Jang and E.K. Ryu,
Point convergence of Nesterov’s Accelerated Gradient Method: an AI-assisted proof, 
arXiv, 2025. \url{https://arxiv.org/abs/2510.23513}

\bibitem{HWWM}
A.~Hamadouche, Y.~Wu, A.~M.~Wallace, and J.~F.~C.~Mota,
Sharper bounds for proximal gradient algorithms with errors,
{\it SIAM J.~Optim.}, 34 (2024), pp.~278--305.

\bibitem{Her25} 
J.~Hermant, M.~Renaud, J.-F.~Aujol, Ch.~Dossal, and A.~Rondepierre,
Gradient Correlation is a key ingredient to accelerate SGD with momentum,  
{\it ICLR}, 2025.

\bibitem{Knopp}
K.~Knopp,
{\it Infinite Sequences and Series}, 
Dover Publications, New York, 1956.

\bibitem{Lan} 
G. Lan, An optimal method for stochastic composite optimization, \emph{Math. Program.}, 133 (2012), 365--397. \url{https://doi.org/10.1007/s10107-010-0434-y}.

\bibitem{Nesterov} 
Y. Nesterov, 
A method of solving a convex programming problem with convergence rate $O(1/{k^2})$, 
{\it Dokl.~Akad.~Nauk SSSR}, 269 (1983), pp.~543--547.

\bibitem{Opial1967}
Z.~Opial, 
Weak convergence of the sequence of successive approximations for nonexpansive mappings,
{\it Bull.~Amer.~Math.~Soc.}, 73 (1967), pp.~591--597.



\bibitem{Pol87}
B.~T.~Polyak,
{\it Introduction to Optimization},
Optimization Software, Publications Division, New York, 1987.

\bibitem{Rav24}
H.~Attouch, J. Fadili, and V. Kungurtsev,
The stochastic Ravine accelerated gradient method with general extrapolation coefficients, 
 arXiv, 2025. \url{https://arxiv.org/abs/2403.04860}


\bibitem{RobSig1971} 
H.~Robbins and D.~Siegmund, 
A convergence theorem for nonnegative almost supermartingales and some applications, in:
{\it Optimization Methods in Statistics}, J.~S.~Rustagi, ed., pp.~233--257, Academic Press, New York, London, 1971.


\bibitem{Salzo-Villa12}
S.~Salzo and S.~Villa,
Inexact and accelerated proximal point algorithms,
{\it J.~Convex Anal.}, 19 (2012), pp.~1167--1192.


\bibitem{Sch11}
M.~Schmidt, N.~Le Roux, and F.~Bach,
Convergence Rates of Inexact Proximal-Gradient Methods for Convex Optimization,
NIPS, 2011.


\bibitem{Sch11_2}
M.~Schmidt, N.~Le Roux, and F.~Bach,
Convergence Rates of Inexact Proximal-Gradient Methods for Convex Optimization,
arXiv, 2011. \url{https://arxiv.org/abs/1109.2415v2}


\bibitem{SuBoCa}
W.~Su, S.~Boyd, E.~J.~Cand\`es,
A differential equation for modeling Nesterov's accelerated gradient method: theory and insights,
{\it JMLR}, 17 (2016), pp.~1--43.


\bibitem{Vas19} 
S.~Vaswani, F.~Bach, and M.~Schmidt,
Fast and Faster Convergence of SGD for Over-Parametrized Models and an Accelerated Perceptron,
{\it AISTATS}, 2019.

\bibitem{Villa-Salzo}
S.~Villa, S.~Salzo, L.~Baldassarre, A.~Verri,
Accelerated and inexact forward-backward algorithms,
{\it SIAM J.~Optim.}, 23 (2013), pp.~1607--1633.


\end{thebibliography}

\appendix
\section{Appendix}
\subsection{Proofs of Sections~\ref{sec:basic} and \ref{sec:analysis}} 
\label{sec:app}
For reader's convenience we give the proofs of selected results presented in the preliminary section, as well as of Remark~\ref{rmk:ass2}\ref{rmk:ass2_v} and  Remark~\ref{rmk:relproxerror}.

\begin{proof}[\textbf{Proof of Fact~\ref{lem:20201030b}}]
Define $u_k = \alpha_k + \sum_{i=k}^{+\infty} \varepsilon_i$. Then it follows from \eqref{eq:20201030a} that
$u_{k+1} = a_{k+1} + \sum_{i={k+1}}^{+\infty} \varepsilon_i \leq \alpha_k + \sum_{i={k}}^{+\infty} \varepsilon_i=u_k$,
so that $(u_k)_{k \in \N}$ is decreasing and hence convergent. Then, by definition of $u_k$,
$ \alpha_k = u_k - \sum_{i=k}^{+\infty} \varepsilon_i$ and hence $(\alpha_k)_{k \in \N}$ is convergent too.
\end{proof}

\begin{proof}[\textbf{Proof of Fact~\ref{lem:Bih-LaS}}]
Set $\gamma_k = \sigma_k + \sum_{i=0}^k \lambda_i \sqrt{\mu_i}$, with $\gamma_k\geq 0$. Then clearly $(\gamma_k)_{k \in \N}$ is increasing and $\sqrt{\mu_k} \leq \sqrt{\gamma_k}$, for every $k \in \N$. Therefore,
\begin{equation*}
\forall\, k \in \N\colon \gamma_k = \sigma_k + \sum_{i=0}^k \lambda_i \sqrt{\mu_i} \leq 
\sigma_k + \sum_{i=0}^k \lambda_i \sqrt{\gamma_i} \leq \sigma_k + \sqrt{\gamma_k}\sum_{i=0}^k \lambda_i.
\end{equation*}
Thus, for all $k \in \N$, $\sqrt{\gamma_k}$ is a positive  solution of the equation
$t^2 - \big(\sum_{i=0}^k \lambda_i\big) t - \sigma_k \leq 0$ and hence
\begin{equation}
\label{eq:20251110a}
    \sqrt{\gamma_{k}} \leq \frac 1 2 \sum_{i=0}^k \lambda_i+ \bigg( \Big(\frac1 2 \sum_{i=0}^k \lambda_i\Big)^2 + \sigma_k\bigg)^{1/2}.
\end{equation}
The statement follows. 
\end{proof}

\begin{proof}[\textbf{Proof of Fact~\ref{lem:radu_enis}}]
Let $b_k = a_{k+1}+ \lambda_k(a_{k+1}-a_k)$. We will prove that $a_k$
can be written as a weighted sum of $b_k$'s terms. Let $(\mu_k)_{k\in \N}$
be sequence in $\R_{++}$ (to be defined). Then, by the definition of $b_k$, we have
\begin{equation*}
\mu_k b_k = \mu_k a_{k+1}+ \mu_k\lambda_k(a_{k+1}-a_k)
= (\mu_k\lambda_k + \mu_k)a_{k+1}- \mu_k\lambda_k a_k.
\end{equation*}
Suppose now that 
\begin{equation}
(\forall\,k\in \N)\quad\mu_{k+1}\lambda_{k+1}=\mu_k\lambda_k + \mu_k.
\end{equation}
Note that the above identity actually defines the sequence $(\mu_{k})_{k \in \N}$
recursively as $\mu_{k+1} = \mu_k(1+\lambda_k)/\lambda_{k+1}$,
starting from any $\mu_0>0$.
Then we have
\begin{equation}
\label{eq:20251105a}
\mu_k b_k =  \mu_{k+1}\lambda_{k+1}a_{k+1}- \mu_k\lambda_k a_k
\quad\text{and}\quad 
\mu_k = \mu_{k+1}\lambda_{k+1}-\mu_k\lambda_k
\end{equation}
and hence, summing
\begin{equation*}
\sum_{i=0}^k \mu_i b_i 
= \mu_{k+1}\lambda_{k+1}a_{k+1} - \mu_{0}\lambda_{0}a_{0}
\quad\text{and}\quad
\sum_{i=0}^k \mu_i = \mu_{k+1}\lambda_{k+1}-\mu_0\lambda_0.
\end{equation*}
Thus, we have
\begin{equation*}
\mu_{k+1}\lambda_{k+1}a_{k+1}=  \mu_{0}\lambda_{0}a_{0} + \sum_{i=0}^k \mu_i b_i
\quad\text{and}\quad
\mu_{k+1}\lambda_{k+1} = \mu_0\lambda_0+\sum_{i=0}^k \mu_i
\end{equation*}
which yield
\begin{equation*}
a_{k+1} = \frac{\mu_{0}\lambda_{0}a_{0} + \sum_{i=0}^k \mu_i b_i}{\mu_0\lambda_0+\sum_{i=0}^k \mu_i} = \sum_{i=0}^{k+1} \frac{w_{i}}{\sum_{j=0}^{k+1}w_j} \tilde{b}_i,
\end{equation*}
where $\tilde{b}_0 = a_0$, $w_{0} = \mu_0\lambda_0$,  and for $k\geq 1$, $\tilde{b}_k = b_{k-1}$
and $w_{k} = \mu_{k-1}$. Now, we prove that $\mu_{k+1}\lambda_{k+1}=\sum_{i=0}^{k+1} w_i\to+\infty$ as $k\to+\infty$. Indeed it follows from the second identity in \eqref{eq:20251105a} that
\begin{equation*}
\mu_{k+1}\lambda_{k+1} = \mu_k \lambda_k\bigg(1+\frac{1}{\lambda_k} \bigg)
\end{equation*}
and hence by recursion we have $\mu_{k+1} \lambda_{k+1} = \mu_0\lambda_0\prod_{i=0}^k\big(1 + 1/\lambda_i \big)$. Finally, we note that 
\begin{equation*}
\prod_{k=0}^{+\infty}\bigg(1 + \frac{1}{\lambda_k}\bigg)=+\infty\ \Leftrightarrow\ 
\sum_{k=0}^{+\infty} \frac{1}{\lambda_k}=+\infty
\end{equation*}
and hence, recalling the assumptions, $\sum_{k=0}^{+\infty} w_k = +\infty$ as claimed.
The statement follows from the Ces\`aro theorem.
\end{proof}

\begin{proof}[\textbf{Proof of Remark~\ref{rmk:ass2}\ref{rmk:ass2_v}}]
We use the same notation of Lemma~\ref{lem:ass2} and set $\widehat{G} = \widehat{G}_\gamma(y,\zeta)$, $G = G_\gamma(y)$,
and $\nu^2 = \EE[\norm{D - \EE D}^2]$. We first note that by the triangle inequality in $\HH$, Cauchy-Schwarz in $\R^2$,
and Lemma~\ref{lem:ass2}\ref{lem:ass2_i},
\begin{align*}
\norm{G}\cdot \nu &\leq \norm{G + \EE D} \cdot\nu + \nu\cdot \norm{\EE D}\leq \sqrt{(\norm{G + \EE D}^2 + \nu^2\big) \cdot (\nu^2 + \norm{\EE D}^2\big)} \\
& = \sqrt{\big(\norm{\EE[\widehat{G}]}^2 + \EE[\norm{\widehat{G}- \EE[\widehat{G}]}^2]\big) \cdot \EE[\norm{D}^2]}
\leq \sqrt{\EE[\norm{\widehat{G}}^2]\cdot \EE[\scalarp{\widehat{G}, \Err}]}.
\end{align*}
Thus,  by Cauchy-Schwarz inequality in $L^2(\Omega, \PP)$ and \eqref{eq:20261014c},
we have
\begin{align*}
\big(\EE[\scalarp{\widehat{G}, \Err}]\big)^2 &= \big(\EE[\scalarp{D - \EE[D], \Err}]\big)^2 \leq \EE[\norm{D - \EE[D]}^2] \EE[\norm{\Err}^2]\\
& \leq \beta^2 \norm{G}^2 \EE[\norm{D - \EE[D]}^2] + \kappa^2 \EE[\norm{D - \EE[D]}^2]\\
&\leq \beta^2 \EE[\norm{\widehat{G}}^2] \EE[\scalarp{\widehat{G}, \Err}] + \kappa^2 \EE[\scalarp{\widehat{G}, \Err}].
\end{align*}
Hence $\EE[\scalarp{\widehat{G}, \Err}] \leq \beta^2 \EE[\norm{\widehat{G}}^2] + \kappa^2$, and if $\beta^2<1$, then Assumption~\ref{ass:2}\ref{ass:2_v}
holds with $\rho = 1/(1-\beta^2)$ and $\sigma = \kappa$. Finally, by \eqref{eq:20261014c}, and Jensen's inequality
\begin{equation}
\label{eq:20261014d}
\EE[\norm{\Err}] \leq (\EE[\norm{\Err}^2])^{1/2} \leq \beta \norm{G} + \kappa.
\end{equation}
Also, by the triangle inequality in $L^2(\Omega, \PP)$, $\norm{G} \leq \norm{D}_{L^2} + \norm{\widehat{G}}_{L^2}
\leq (\EE[\scalarp{\widehat{G}, \Err}])^{1/2} + (\EE[\norm{\widehat{G}}^2])^{1/2} \leq \beta (\EE[\norm{\widehat{G}}^2])^{1/2} + \kappa + (\EE[\norm{\widehat{G}}^2])^{1/2}$. Hence $\norm{G} \leq (1+\beta) (\EE[\norm{\widehat{G}}^2])^{1/2} + \kappa$. This last inequality and \eqref{eq:20261014d} gives $\EE[\norm{\Err}] \leq \beta(1+\beta)(\EE[\norm{\widehat{G}}^2])^{1/2} + \kappa (1+\beta)$.
\end{proof}

\begin{proof}[\textbf{Proof of Remark~\ref{rmk:relproxerror}}]
Following the same line of arguments as in the proof of Theorem~\ref{mainthm}, we now have
\begin{equation*}
y_{k+1} = \bigg(1 - \frac{1}{t_{k+1}} \bigg) x_{k+1} + \frac{1}{t_{k+1}}
\big(\underbrace{x_k + t_k(x_{k+1} - x_k) - t_k e_k}_{=:v_{k+1}}\big),
\end{equation*}
so that
\begin{equation}
\label{eq:20261009h}
y_k = \bigg(1 - \frac{1}{t_k} \bigg) x_k + \frac{1}{t_k}v_k,\qquad
x_{k+1} = \bigg(1 - \frac{1}{t_k} \bigg) x_k + \frac{1}{t_k}(v_{k+1} + t_k e_k),
\end{equation}
where $v_0=y_0$.
Now, invoking Lemma~\ref{lem:20190313d}, with $x = y_k$, and using the second inequality in \eqref{eq:relproxerror}, we obtain, for every $z\in \HH$,
\begin{align*}
2\gamma (F(x_{k+1})- F(z))
&\leq  \norm{z - y_k}^2 - \norm{x_{k+1} - z}^2 - (1-\gamma L) \norm{x_{k+1}-y_k}^2\\[1ex]
&\hspace{13ex}+ 2\langle x_{k+1}-z, e_k\rangle+\delta_{1,k}^2\\
&=  \norm{z - y_k}^2 - \norm{x_{k+1} - e_k - z}^2 - (1-\gamma L) \norm{x_{k+1}-y_k}^2\\[1ex]
&\hspace{13ex}+ \norm{e_k}^2+\delta_{1,k}^2\\
&\leq   \norm{z - y_k}^2 - \norm{x_{k+1} - e_k - z}^2 - \eta \norm{x_{k+1}-y_k}^2.
\end{align*}
Let $x_* \in S_*$ and make the same choice of $z$ as in \eqref{eq:20261009e}. Then
$y_k - z = (v_k - x_*)/t_k$ and $x_{k+1}- e_k - z = (v_{k+1}- x_*)/t_k$ and using the convexity of $F$ one finally obtains 
\begin{equation*}
2 \gamma t_k^2 r_{k+1} + \norm{v_{k+1}- x_*}^2 + \eta t_k^2 \norm{x_{k+1}-y_k}^2 \leq 2 \gamma t_k(t_k - 1) r_k + \norm{v_k - x_*}^2.
\end{equation*}
Thus, defining as before $\mathcal{E}_0(x_*)=\norm{v_0 - x_*}^2$ and $\mathcal{E}_k(x_*) = 2 \gamma t^2_{k-1} r_k + \norm{v_k - x_*}^2$, for $k\geq 1$, the following recursion holds
\begin{equation}
\label{eq:20261009g}
(\forall\,k\in \N)\quad
\mathcal{E}_{k+1}(x_*) + \eta  t_k^2\norm{x_{k+1}- y_k}^2\leq \mathcal{E}_k(x_*).
\end{equation}
It follows that $\mathcal{E}_k(x_*)$ is decreasing and convergent. In particular, $(v_k)_{k \in \N}$
is bounded and $2 \gamma t^2_{k} (F(x_{k+1})- F_*) \leq \mathcal{E}_{k+1}(x_*) \leq \norm{x_0 - x_*}^2$, which yields \eqref{eq:relproxrate}.
Moreover, by \eqref{eq:20261009g} we derive that
\begin{equation*}
\sum_{k=0}^{+\infty} \eta t_k^2 \norm{x_{k+1}-y_k}^2 \leq \mathcal{E}_0(x_*)<+\infty.
\end{equation*}
Now, if $\eta>0$, combining this bound with the second inequality in  \eqref{eq:relproxerror}, we have
$\sum_{k} t_k^2(\norm{e_k}^2 + \delta^2_{1,k})<+\infty$ and hence $t_k e_k \to 0$. If instead $e_k=0$ for every $k\in \N$, the same limit holds directly.
Thus, in either case, $(v_{k+1} + t_k e_k)_{k \in \N}$ is bounded. The second identity in \eqref{eq:20261009h} now gives
$\norm{x_{k+1}- x_*} \leq \max\{\norm{x_k - x_*}, \norm{v_{k+1}+ t_k e_k - x_*}\}$,
which by induction provides the boundedness of $(x_k)_{k \in \N}$ too. The argument can continue as in the proof of Theorem~\ref{mainthm},
proving that for any $x^1_*, x^2_*\in S_*$, the sequence $(\scalarp{v_k, x^1_* - x^2_*})_{k \in \N}$ converges and noting that because of the second identity in \eqref{eq:20261009h}, 
\begin{equation*}
\scalarp{v_{k+1}+ t_k e_k, x^1_*- x^2_*} = \scalarp{x_{k+1}, x^1_*- x^2_*} + (t_k - 1)\big(\scalarp{x_{k+1}, x^1_*- x^2_*} - \scalarp{x_{k}, x^1_*- x^2_*} \big),
\end{equation*}
where the left hand side converge, since $t_k e_k \to 0$.
\end{proof}

\end{document}